\documentclass{amsart}
\usepackage{amsmath,amssymb,enumerate,verbatim}
\usepackage{amsmath, amssymb}
\usepackage{srcltx}
\usepackage{amsthm}
\usepackage{lipsum}
\usepackage{comment}
\usepackage{float}

\usepackage{booktabs}
\usepackage{longtable}
\usepackage{array}
\usepackage{pdflscape}
\usepackage[colorlinks=true,
            linkcolor=blue,
            citecolor=blue,
            urlcolor=blue]{hyperref}

\newif\ifshowoldtable
\showoldtablefalse

\usepackage[usenames,dvipsnames]{xcolor}
\usepackage{graphicx}
\usepackage{multirow}
\usepackage{booktabs}
\usepackage{mathtools,bm}
\usepackage{blkarray}
\usepackage{tikz} 
\usepackage{multirow}
\usepackage{array,longtable}
\usetikzlibrary{matrix,arrows}
\usepackage{dsfont}
\usepackage{float}
\usepackage{perpage}
\usepackage{mathrsfs}
\usepackage[scr=boondoxo]{mathalpha}

\DeclareMathOperator{\Lift}{Lift}
\newcommand{\LLift}{\mathscr{Lift}}

\usepackage[all]{xy}

\newcommand{\atlas}{\texttt{atlas}}
\newcommand{\LiftC}{\Lift_G^{\tu G}(\bbC)}

\newcommand{\ep}{{\epsilon}}

\newcommand{\bbC}{{\mathbb{C}}}

\newcommand{\bbR}{{\mathbb R}}

\newcommand{\bbZ}{{\mathbb Z}}

\newcommand{\cal}{\mathcal}

\newcommand{\fk}{\mathfrak}
\newcommand{\ovl}{\overline}

\newcommand{\Spin}{{\mathrm{Spin}}}
\newcommand{\SL}{{\mathrm{SL}}}
\newcommand{\GL}{{\mathrm{GL}}}

\newcommand{\sgn}{{\mathrm{sgn}}}

\newcommand{\Ind}{{\mathrm{Ind}}}

\newcommand{\SU}{\mathrm{SU}}
\newcommand{\adm}{\mathrm{adm}}

\newtheorem{thm}{Theorem}[section]
\newtheorem{prop}[thm]{Proposition}
\newtheorem{lemma}[thm]{Lemma}
\newtheorem{cor}[thm]{Corollary}
\newtheorem{remark}[thm]{Remark}

\theoremstyle{definition}
\newtheorem{definition}[thm]{Definition}
\newtheorem{notation}[thm]{Notation}

\newcommand{\calB}{{\mathcal{B}}}

\newcommand{\calF}{{\mathcal{F}}}

\newcommand{\calO}{{\mathcal{O}}}

\newcommand{\calR}{{\mathcal{R}}}

\newcommand{\tu}{\widetilde}

\begin{document}

%\subjclass{22E46, 22E47}

\title{Lift of the Trivial Representation:\\ The Nonsplit Type $D$ Case}

\author{Wan-Yu Tsai}

 \thanks{The author is supported by the National Science and Technology Council of Taiwan under grant no. 114-2115-M-008-003-MY3.}

\curraddr{
Department of Mathematics\\
National Central University\\
Taoyuan 320 \\
Taiwan} 
\email{wytsai@math.ncu.edu.tw}

\begin{abstract}

Let $\widetilde G$ be a nonlinear double cover of the real points of a
connected, simply connected, semisimple complex group.
In \cite{Ts19}, we introduced
$\prod_{\rho/2}^s(\tu G)$, the set of genuine small representations
with infinitesimal character $\rho/2$ in the simply laced setting.
In \cite{Ts23}, we showed that, when $G$ is split, these representations
are exhausted by the Kazhdan--Patterson lift of the trivial representation.  In this paper, we consider nonsplit real Spin groups of type $D_n$. For each
$\rho/2$-compatible genuine central character, we determine the irreducible
constituents of
$\Lift_G^{\widetilde G}(\bbC)$.
We also describe explicitly the relevant genuine small representations in terms
of Langlands parameters. The nonsplit case exhibits a new phenomenon: the lift
of the trivial representation need not exhaust
$\prod\nolimits^s_{\rho/2}(\widetilde G).
$ 
For the families considered here, equality holds only for
\(\widetilde{\Spin}(2p+2,2p)\); in the other two families the lift is a proper
subset, and in one case it is zero.

\end{abstract}

\maketitle

\section{Introduction}

Assume that $G_\bbC$ is a connected, simply connected, semisimple complex
algebraic group, and let $G$ be a real form of $G_\bbC$ with nontrivial
fundamental group. Then $G$ admits a nonlinear double cover $\tu G$, which is
not a matrix group; see \cite[Proposition 3.6]{AT12}. In fact, most real forms
of $G_\bbC$ admit nonlinear double covers (see \cite{A04}). Examples include
$\SL(n,\bbR)$, $\SU(p,q)$, $\Spin(p,q)$, and most real forms of exceptional
groups.
\medskip

The representation theory of nonlinear covering groups has also been
developed extensively over nonarchimedean local fields, particularly in
the framework of metaplectic and Brylinski--Deligne covering groups;
see, for example, \cite{BD01,GGW18,Wei18,GG18} and the references therein.
Although the archimedean and nonarchimedean theories involve rather different techniques, several common themes include lifting correspondences, small representations, and their connections with nilpotent orbits. These parallel themes provide part of the broader motivation for studying the lift of the trivial representation and its small constituents in the present paper.

\medskip 

In this paper, we study a basic source of genuine representations of nonlinear double covers of real reductive groups: the lift of the trivial representation. In previous work, we showed that for split simply laced groups this lift produces all genuine small representations attached to a particular nilpotent orbit. The nonsplit case turns out to behave quite differently. For nonlinear double covers of real Spin groups of type $D$, we show that, depending on the real form, the lift may exhaust the corresponding set of small representations, produce only a proper subset, or vanish entirely. Thus these groups provide a concrete family in which the interaction among lifting, genuine small representations, and nilpotent geometry can be seen explicitly.

\medskip

The present work is a continuation of \cite{Ts19} and \cite{Ts23}. Our main object is
the set of irreducible genuine representations occurring in the
Kazhdan--Patterson lift of the trivial representation. We denote this set by
$$
\LLift(\bbC).
$$
Here the lifting operator is understood in the sense of Adams--Herb
\cite{AHe10}: for a stable admissible representation $\pi$ of $G$, the lift
$\Lift_G^{\tu G}(\pi)$ is a genuine virtual representation of $\tu G$, or zero.
Thus we may write
\[
\Lift_G^{\tu G}(\pi)=\sum_{\tu\pi} a_{\tu\pi}\tu\pi,
\qquad a_{\tu\pi}\in\bbZ,
\]
and define
\[
\LLift(\pi)=\{\tu\pi\mid a_{\tu\pi}\neq 0\}.
\]

The lifting of characters for nonlinear groups was first developed by
Kazhdan and Patterson for \(\GL(n)\) in \cite{KP84,KP86}.  For general
linear groups over the archimedean fields, Tadi\'c treated
\(\GL(n,\bbC)\) in \cite{T96}, while Adams and Huang gave a detailed analysis for
\(\GL(n,\bbR)\) in \cite{AHu97}.  The Shimura lifts of pseudospherical
principal series representations were studied in \cite{ABPTV07}.
More recently, we have also studied the lifts of certain special
unipotent representations for complex and real groups in the preprint
\cite{TWZ} and a forthcoming work \cite{TsW}, respectively.
\medskip

Related Shimura correspondences for covering groups over nonarchimedean
local fields have been studied from several points of view.  Savin's
local Shimura correspondence compares genuine representations of a
covering group with representations of a linear group through
Iwahori--Hecke algebras \cite{Sa88}.  For metaplectic groups, Gan--Savin
established Hecke algebra correspondences between the Bernstein
components containing the even and odd Weil representations and
Iwahori-spherical components of odd orthogonal groups
\cite{GS12}; this correspondence was extended to residual
characteristic two by Takeda--Wood \cite{TW18}.  More recently, Wang
has developed type-theoretic and affine Hecke algebra methods for tame
genuine principal-series components, obtaining a broader form of local
Shimura correspondence \cite{WangShimura}.  These results provide
another manifestation of the general principle that genuine
representations of nonlinear covering groups may be related to
representations of linear groups through lifting or Shimura-type
correspondences.
\medskip

The present series of papers is especially motivated by the lifting of
one-dimensional representations of $\GL(n,\bbR)$.  In particular,
the lifts of the trivial and sign representations are described
explicitly in terms of the representations $T_n$, $T_n(\chi_0)$,
and $\operatorname{Speh}(1/2)$, depending on the parity of $n$;
see \cite{H90,AHu97}.  These results motivate the analogous question
for other simply laced real groups: we seek to determine the
irreducible constituents of the lift of a one-dimensional
representation.
\medskip

In the present sequel, $G$ is the group of real points of a connected,
simply connected, semisimple, simply laced complex group, and $\tu G$
is the nonlinear double cover of $G$.  Under these assumptions, the
only one-dimensional representation of $G$ is the trivial
representation $\bbC$.  Thus
$\Lift_G^{\tu G}(\bbC)$
provides a natural source of genuine representations of $\tu G$.

\medskip
The representations arising from such lifts are closely related to
small genuine representations.  One example is  the
irreducible quotient of the pseudospherical principal series studied
in \cite{ABPTV07}.  There are parallel constructions over
nonarchimedean local fields, most notably the theta
representations of covering groups.  These representations have been
studied extensively from the viewpoints of Whittaker models, Fourier
coefficients, and nilpotent orbits; see, for example,
\cite{KP84,GRS97,Gao17,Cai19,GT22,GLT25}.  For related constructions
and broader discussions of small representations, see also
\cite{BFG03,Kap17}.  An important theme in this literature is the
determination of wavefront sets of theta representations.  The
resulting nilpotent orbits exhibit a close parallel with the associated
varieties arising in the archimedean theory.

\medskip
More precisely, in \cite{Ts19} we introduced a set
$
\prod\nolimits_\lambda^s(\tu G)
$
of genuine small representations with infinitesimal character
$\lambda$.  For simply laced groups, the relevant infinitesimal
character is $\lambda=\rho/2$.  These small representations are
attached to a certain complex nilpotent orbit $\calO$, and they are
shown to be precisely the genuine representations with complex
associated variety $\calO$.  These small representations also have
important algebraic properties.  For example, their $K$-types are
multiplicity-free, and they include the representations obtained from
the trivial representation by the Shimura lift; see \cite{ABPTV07}.

\medskip

For the present paper, the important consequence is that, when $G$ is
simply laced,
$$
\LLift(\bbC)\subseteq \prod\nolimits_{\rho/2}^s(\tu G).
$$
Thus the set $\prod\nolimits_{\rho/2}^s(\tu G)$ provides a concrete
list of possible constituents of the lift.

The nilpotent orbit $\calO$ imposes strong restrictions on the real
form $G$.  Indeed, if
$$
\calO\cap\fk g_0=\emptyset,
$$
then
$$
\Lift_G^{\tu G}(\bbC)=0,
$$
where $\fk g_0$ is the Lie algebra of $G$.  Hence it is enough to
consider the real forms for which
\begin{equation}\label{e:nonempty-cap}
\calO\cap\fk g_0\neq\emptyset.
\end{equation}
These real forms are listed in \cite[Table~2]{Ts19}.  Most of them are
split or quasi-split, with only a few nonsplit exceptions.

In \cite{Ts23}, we treated the split case.  When $G$ is split of
simply laced type, we proved that
$$
\LLift(\bbC)=
\prod\nolimits_{\rho/2}^s(\tu G).
$$
Equivalently, every genuine small representation with infinitesimal
character $\rho/2$ occurs in the lift of the trivial representation.

According to \cite[Table 2]{Ts19}, the remaining nonsplit real forms satisfying
\eqref{e:nonempty-cap} are
\begin{align*}
A_{2m-1}: \ & \SU(m,m),\\
A_{2m}: \ & \SU(m+1,m),\\
D_n:\  & \Spin(n+1,n-1),\ \Spin(n+2,n-2),\\
E_6: \ & E_{6(2)}, \ \text{the quasi-split real form of } E_6(\bbC).
\end{align*}
For type \(D_n\), however, the condition \eqref{e:nonempty-cap} is satisfied
by \(\Spin(n+2,n-2)\) only when \(n\) is odd. Indeed, the complex nilpotent
orbit \(\calO\) under consideration is
\[
[3\,2^{n-2}\,1]
\quad\text{if \(n\) is even,}
\]
and
\[
[3\,2^{n-3}\,1^3]
\quad\text{if \(n\) is odd.}
\]
Thus, in type \(D\), the real Spin groups satisfying \eqref{e:nonempty-cap}
are the split groups, together with the following nonsplit families:
\[
\Spin(2p+1,2p-1),\quad
\Spin(2p+2,2p),\quad
\Spin(2p+3,2p-1).
\]
The present paper completes the nonsplit type \(D\) case. The type \(A\) and
type \(E\) cases will be treated elsewhere.

Thus, throughout this paper, we consider
\[
\tu G\in
\{\tu{\Spin}(2p+1,2p-1),
\tu{\Spin}(2p+2,2p),
\tu{\Spin}(2p+3,2p-1)\}.
\]
For the unequal-rank families
$\Spin(2p+1,2p-1)
\text{ and } 
\Spin(2p+3,2p-1),$
we assume \(p\geq 2\). For the equal-rank family $\Spin(2p+2,2p),$
we allow \(p\geq 1\).

We now state the main result. The notation used in the theorem is introduced
precisely in later sections. Briefly,
\[
\calR_{\rho/2}(\tu G)_\chi
\]
denotes the explicitly constructed set of genuine small representations with
infinitesimal character $\rho/2$ and genuine central character $\chi$, and
\[
\LLift(\bbC)_\chi
=
\{\tu\pi\in\LLift(\bbC)\mid \tu\pi
\text{ has central character } \chi\}.
\]

\begin{thm} \label{t:main1}\textup{(See Theorem~\ref{t:main})}
Using the notation in Definition \ref{d:list}, and fixing
$\chi\in \widehat Z_{\rho/2}(\tu G)$, the set $\LLift(\bbC)_\chi$ is described
as follows.
\begin{itemize}
    \item[(a)] If $\tu G=\tu{\Spin}(2p+1,2p-1)$, then
    \[
    \LLift(\bbC)_\chi
    =\begin{cases}
        \emptyset, & \ \chi =\chi_1\\
    \{J(\{2p\}^*)\}, & \chi=\chi_2.
    \end{cases}
    \]

    \item[(b)] If $\tu G=\tu{\Spin}(2p+2,2p)$, then
    \[
    \LLift(\bbC)_\chi
    =
    \calR_{\rho/2}(\tu G)_\chi.
    \]

    \item[(c)] If $\tu G=\tu{\Spin}(2p+3,2p-1)$, then
    \[
    \LLift(\bbC)_\chi
    =
    \emptyset.
    \]
\end{itemize}
Here the representations appearing on the right-hand side are understood to
have central character $\chi$.

Consequently,
\[
\LLift(\bbC)
=
\prod\nolimits_{\rho/2}^{s}(\widetilde G)
\]
if and only if
\(\widetilde G=\widetilde{\Spin}(2p+2,2p)\).
For \(\widetilde{\Spin}(2p+1,2p-1)\), the lift exhausts the
\(\chi_2\)-block but vanishes on the \(\chi_1\)-block, whereas for
\(\widetilde{\Spin}(2p+3,2p-1)\) the lift is empty.

\end{thm}

We briefly describe the structure of the paper. In Section 2, we recall the
necessary notation and structure theory, and we introduce the parametrization of
regular characters used throughout the paper. In Section 3, we construct
explicit genuine small representations in
$\prod\nolimits_{\rho/2}^s(\tu G)$
by computing their $\tau$-invariants. In Section 4, we use coherent
continuation to count the number of representations in
$\prod\nolimits_{\rho/2}^s(\tu G)$
and show that the representations constructed in Section 3 exhaust the set.

In Section 5, we study the lift $\Lift_G^{\tu G}(\bbC)$ and prove Theorem
\ref{t:main1}, the main theorem in this paper. The small representations constructed in the previous sections
are precisely the possible constituents of the lift, but in contrast with the
split case, the lift need not exhaust
$\prod\nolimits_{\rho/2}^s(\tu G).$
After establishing the general lifting theorem, we give a detailed
calculation for \(\tu{\Spin}(5,3)\) in Section 6, which illustrates the main features
of the lifting argument in the first substantial low-rank case.  In the
final section, following \cite{BTs18}, we record further properties of the resulting
representations, including their  \(\tu K\)-types, real associated
varieties, and unitarity.

\section{Structure Theory}

\subsection{Preliminaries}

We recall some notation for regular characters, cross actions, Cayley transforms, and genuine central characters. For the general definitions, we refer the reader to \cite[Section 2]{Ts23} and the references therein.

Let $G$ be a real form of a connected, simply connected, semisimple complex Lie group $G_{\mathbb C}$ of simply laced type, and let $\widetilde G$ be the nonlinear double cover of $G$. We identify the kernel of the covering map $p:\widetilde G\to G$ with $\{\pm 1\}$. If $H$ is a subgroup of $G$, we write $\widetilde H=p^{-1}(H)$.

Let $\mathfrak g_0$ be the Lie algebra of $G$, and let $\mathfrak g$ be its complexification. Fix a Cartan involution $\theta$. Let $H$ be a $\theta$-stable Cartan subgroup of $G$, and let $\mathfrak h$ be the complexified Lie algebra of $H$. For a regular infinitesimal character $\lambda\in \mathfrak h^*$, we write $\Delta^+(\lambda)$ for the positive root system of $\Delta(\mathfrak g,\mathfrak h)$ making $\lambda$ dominant.

The integral root system defined by $\lambda$ is
\begin{equation*}
R(\lambda)=\{\alpha\in\Delta(\mathfrak g,\mathfrak h)\mid
\langle \lambda,\alpha^\vee\rangle\in\mathbb Z\}.
\end{equation*}
We write
\begin{equation*}
R^+(\lambda)=R(\lambda)\cap \Delta^+(\lambda),
\qquad
W(\lambda)=W(R(\lambda))
\end{equation*}
for the positive integral roots and the integral Weyl group, respectively.
\medskip

Let $\mathcal{HC}(\fk g, K)$ be the set of Harish-Chandra modules and let  $\cal{HC}(\fk g, K)_\lambda \subset \cal{HC}(\fk g, K)$ be the set of Harish-Chandra modules 
with infinitesimal character $\lambda$.  The set of equivalence classes of irreducible admissible representations of $G$, denoted $\widehat G _{\adm}$, can be regarded as 
a subset of $\cal{HC}(\fk g, K)$ by sending an irreducible admissible representation to its space of $K$-finite vectors. Similarly, $\widehat G_{\adm,\lambda}$ is denoted the set 
of  equivalence classes of irreducible admissible representations of $G$ with infinitesimal character $\lambda$.  The same notions will also be used for the nonlinear group $\tu G$. An
irreducible representation $\pi$ of $\tu G$ is \emph{genuine} if $\pi(-\tu g)=-\pi(\tu g)$ for all $\tu g\in \tu G$.

Every $\pi\in \widehat G_{\adm,\lambda}$ is  specified by a parameter, which is called a $\lambda$-regular character,  $\gamma=(H,\Gamma,\overline{ \gamma})$, where $H$ is a $\theta$-stable Cartan subgroup of $G$, $\Gamma$ is a character of $H$, and $\overline{\gamma}$ is an element in $\mathfrak{h}^{\ast}$ which defines the same infinitesimal character as $\lambda$, and there are  certain compatibility conditions between $ \overline{\gamma}$  and $\Gamma$ (see Definition 5.3 in \cite{AT12}). Write $H=TA$, where $T=H^{\theta}$ and $A$ is the identity component of $\{ h\in H \mid \theta (h)=h^{-1}\}$. Let $M=\operatorname{Cent}_G(A)$. The conditions on $\gamma$ imply that there is a unique relative discrete series representation of $M$, denoted by $\sigma_M$,
with Harish-Chandra parameter $\overline\gamma$, whose lowest $M\cap K$-type has $\Gamma$ as a highest weight. Then we define a parabolic subgroup $P=MN$ such that $\pi=J(\gamma)$, the unique irreducible quotient of a standard representation $I(\gamma)= \Ind _P ^G (\sigma _M \otimes 1)$, which is parametrized by $\gamma$ from a $K$-conjugacy class of regular characters for $\lambda$.

A genuine $\lambda$-regular character of $\widetilde G$ is a triple
\begin{equation*}
\tu\gamma=(\widetilde H,\Gamma,\overline\gamma),
\end{equation*}
where $\widetilde H$ is a Cartan subgroup of $\widetilde G$, 
 $\Gamma$ is an irreducible genuine representation of $\tu H$
(equivalently, it may be specified by its genuine central character
on $Z(\tu H)$), and $\overline\gamma\in\mathfrak h^*$ defines the infinitesimal character $\lambda$, subject to the usual compatibility conditions. We denote by $I(\tu\gamma)$ the corresponding standard representation, and by $J(\tu \gamma)$ its unique irreducible quotient.

\medskip

Recall that, when $\lambda$ is regular, the irreducible objects in
$\mathcal{HC}(\fk g,K)_\lambda$ are parametrized by the set
$\mathcal P_\lambda$ of $K$-conjugacy classes of $\lambda$-regular characters.
For $\gamma\in\mathcal P_\lambda$, we write
$I(\gamma)$
for the corresponding standard module, and
$J(\gamma)$
for its unique irreducible quotient.

The classes
\[
\{[I(\gamma)]\}_{\gamma\in\mathcal P_\lambda}
\qquad\text{and}\qquad
\{[J(\gamma)]\}_{\gamma\in\mathcal P_\lambda}
\]
are both bases of the Grothendieck group of
$\mathcal{HC}(\fk g,K)_\lambda$.

\begin{definition}\label{KLVpoly}
For $\gamma,\delta\in\mathcal P_\lambda$, define integers
$M(\gamma,\delta)$ by the character formula
\begin{equation*}
[J(\delta)]
=
\sum_{\gamma\in\mathcal P_\lambda}
M(\gamma,\delta)[I(\gamma)].
\end{equation*}
%We call $M(\gamma,\delta)$ the Kazhdan--Lusztig--Vogan coefficient.

Let $m(\gamma,\delta)$ be the entries of the inverse change-of-basis matrix,
defined by
\begin{equation*}
[I(\delta)]
=
\sum_{\gamma\in\mathcal P_\lambda}
m(\gamma,\delta)[J(\gamma)].
\end{equation*}
Thus
\begin{equation*}
m(\gamma,\delta)=[I(\delta):J(\gamma)]
\end{equation*}
is the composition multiplicity of $J(\gamma)$ in $I(\delta)$.

When $G$ is linear, the coefficients $M(\gamma,\delta)$ are computed by the
Kazhdan--Lusztig--Vogan algorithm (see \cite{V83a}). 
For nonlinear covers, the same change-of-basis notation will be used in the
Grothendieck group of a genuine block (see \cite{RT00} and \cite{RT05}).
\end{definition}

We define the length of a regular character $\gamma$.

\begin{definition}\label{d:length}
Let $\gamma=(H,\Gamma,\ovl\gamma)$ be a regular character, and write
\[
\fk h_{\bbR}=\fk t_{\bbR}\oplus \fk a_{\bbR}
\]
for the Cartan decomposition of the real Lie algebra of $H$. Let
$\theta_\gamma$ denote the induced involution on
$\Delta(\fk g,\fk h)$. The length of $\gamma$ is
\begin{equation*}
\ell(\gamma)
=
\frac{1}{2}
\left|
\{\alpha\in\Delta^+(\ovl\gamma)
\mid
\theta_\gamma(\alpha)\notin\Delta^+(\ovl\gamma)\}
\right|
+
\frac{1}{2}\dim\fk a_{\bbR}.
\end{equation*}
\end{definition}

The definitions  above apply to $\tu G$. 
\medskip

We shall also use the following facts repeatedly.

\begin{prop}\label{prop:genuine-character-determined}
Assume that $G$ is a real form of a connected, simply connected, semisimple complex Lie group $G_{\mathbb C}$, and that $\widetilde G$ is the nonlinear double cover of $G$. In addition, suppose that $G$ is simply laced. Let $H$ be a Cartan subgroup of $G$, and let $H^0$ be the identity component of $H$. Then the following hold.
\begin{itemize}
\item[(1)] \textup{(\cite[Proposition 4.7]{AHe10})} We have
\begin{equation*}
Z(\widetilde H)=Z(\widetilde G)\widetilde H^0.
\end{equation*}
In particular, a genuine character of $Z(\widetilde H)$ is determined by its restriction to $Z(\widetilde G)$ and its differential.

\item[(2)] \textup{(\cite[Proposition 5.5]{AT12})}
A genuine regular character
$\tu\gamma=(\widetilde H,\Gamma,\overline\gamma)$
of $\widetilde G$ is determined by $\overline\gamma$ and the restriction of $\Gamma$ to $Z(\widetilde G)$. Hence the irreducible representation $J(\tu\gamma)$ is also determined by these data.
\end{itemize}
\end{prop}

Let $\chi$ be a genuine central character of $\widetilde G$. We denote by $\mathcal B_{\lambda,\chi}$ the set of equivalence classes of genuine $\lambda$-regular characters with central character $\chi$. Thus, if $\tu\gamma=(\widetilde H,\Gamma,\overline\gamma)$, then $\tu\gamma\in B_{\lambda,\chi}$ means that
\begin{equation*}
\Gamma|_{Z(\widetilde G)}=\chi.
\end{equation*}
By Proposition \ref{prop:genuine-character-determined}, once $\lambda$ and $\chi$ are fixed, a genuine regular character is determined by the Cartan subgroup $\widetilde H$ and the element $\overline\gamma$.

We will use the cross action of the Weyl group on regular characters. If $\tu\gamma\in B_{\lambda,\chi}$ and $w\in W(\lambda)$, then
\begin{equation*}
w\times \tu\gamma\in B_{\lambda,\chi}.
\end{equation*}
Indeed, the cross action by $W(\lambda)$ preserves the restriction of the character to $Z(\widetilde G)$.

We will also use Cayley transforms and inverse Cayley transforms. Suppose that $\alpha$ is a noncompact imaginary root for $\tu\gamma$. Then $c^\alpha(\tu\gamma)$ denotes the Cayley transform of $\tu\gamma$ by $\alpha$. For simplicity, $c^{\alpha}(\tu\gamma)$ is also denoted $\tu\gamma^\alpha$.  Suppose instead that $\alpha$ is a real root for $\tu\gamma$. Then $c_\alpha(\tu\gamma)$ denotes the inverse Cayley transform of $\tu \gamma$ by $\alpha$. (Similarly, we will also use $\tu\gamma_\alpha$ to denote $c_\alpha(\tu\gamma)$.) 
If $S=\{\alpha_1,\ldots,\alpha_l\}$ is a set of strongly orthogonal roots for which the corresponding transforms are defined, we write
\begin{equation*}
c^S(\tu\gamma)=c^{\alpha_1}\cdots c^{\alpha_l}(\tu\gamma),
\qquad
c_S(\tu\gamma)=c_{\alpha_1}\cdots c_{\alpha_l}(\tu\gamma).
\end{equation*}
The Cayley transforms used below preserve the fixed genuine central character. Therefore, once $\chi$ is fixed, all parameters obtained from a seed parameter in $B_{\lambda,\chi}$ by the cross action or by the Cayley transforms considered in this paper remain in $B_{\lambda,\chi}$.

In the rest of this paper, we take $\lambda=\rho/2$.

\subsection{Genuine central characters}
\label{s:central-char}
In this subsection, we first summarize the facts on the genuine central characters of the groups $\widetilde{G}=\tu\Spin(c,d)$ with $c+d=2n$, which have the maximal compact subgroup $\tu K = \Spin(c)\times \Spin (d)$. The center $Z(\widetilde{G})$ is computed in \cite{BTs18}. We recall this as follows.

\begin{lemma} \textup{(\cite[Lemma 6.1]{BTs18})} \label{l:center}
Consider $\widetilde{G}=\tu\Spin(c,d)$ with $c+d=2n$. 
\begin{itemize}
    \item[(1)] When $c=2p$, $d=2q$, $Z(\tu G) \cong
    \begin{cases}
\bbZ_2\times\bbZ_4 & \text{ if at least one of $p$ and $q$ is odd},\\
\bbZ_2\times \bbZ_2\times \bbZ_2 &\text{ otherwise}.
    \end{cases}$
    \item[(2)] When $c=2p+1$, $d=2q+1$, $Z(\tu G)\cong \bbZ_2\times \bbZ_2$. 
\end{itemize}

\end{lemma}

\begin{definition}
Let $\widehat Z_{\mathrm{gen}}(\widetilde G)$ denote the set of genuine characters of $Z(\widetilde G)$. We say that $\chi\in \widehat Z_{\mathrm{gen}}(\widetilde G)$ is $\rho/2$-compatible if $\mathcal B_{\rho/2,\chi}\neq \emptyset$. Equivalently, there exists a genuine $\rho/2$-regular character
\begin{equation*}
\tu\gamma=(\widetilde H,\Gamma,\overline\gamma)
\end{equation*}
such that $\Gamma|_{Z(\widetilde G)}=\chi$.
We denote by $\widehat Z_{\rho/2}(\widetilde G)$ the set of $\rho/2$-compatible genuine central characters.

\end{definition}
\medskip

\begin{remark} \label{r:hw}
For $c=2p$, $d=2q$; or $c=2p+1$, $d=2q+1$, let 
$$
\mu=(a_1,\dots, a_p\mid b_1,\dots, b_q) 
$$
be a $\tu K$-type parametrized by its highest weight, and let $\chi$ be the restriction of 
the highest weight of $\mu$ to $Z(\tu G)$. Then $\chi \in  \widehat Z_{\mathrm{gen}}(\widetilde G)$ iff 
$$
a_i\in \bbZ \text{ and } b_j\in \bbZ+\frac{1}{2}; \text{ or }  
a_i\in \bbZ+\frac{1}{2} \text{ and } b_j\in \bbZ.
$$
\end{remark}

The following lemma records the genuine central characters in the
equal-rank and unequal-rank cases.  It is a reformulation of
\cite[Lemma~6.3]{BTs18}.

\begin{lemma}\textup{(cf. \cite[Lemma~6.3]{BTs18})}
 \label{l:central-char}
Let $\mu_j$ be the $\tu K$-type parametrized by its highest weight:
\begin{align*}
    \mu_1 = (\underbrace{1/2,\dots,1/2}_p \mid \underbrace{0,\dots, 0}_q),\ 
     &\mu_2 = (\underbrace{0,\dots, 0}_p \mid \underbrace{1/2,\dots, 1/2}_q),\\
     \mu_3 = (\underbrace{1/2,\dots,1/2, -1/2}_p \mid \underbrace{0,\dots, 0}_q),\ 
     &\mu_4 = (\underbrace{0,\dots, 0}_p \mid \underbrace{1/2,\dots, 1/2, -1/2}_q)
\end{align*}
Let $\chi_j$ be the restriction of the highest weight of $\mu_j$ to $Z(\tu G)$.
\begin{itemize}
    \item[(1)] If $c=2p$, $d=2q$, then $$\widehat Z_{\mathrm{gen}}(\widetilde G)=\{\chi_1, \chi_2, \chi_3, \chi_4\}.$$ Moreover, given any  $\tu K$-type $\mu$, the central character of $\mu$
is $\chi_j$ iff $\mu-\mu_j \in Q(D_n)$, where $Q(D_n)$ denotes the root lattice of $D_n$.

      \item[(2)] If $c=2p+1$, $d=2q+1$, then $$\widehat Z_{\mathrm{gen}}(\widetilde G)=\{\chi_1, \chi_2\}.$$
  Moreover,  if $\mu=(a_1,\ldots,a_p\mid b_1,\ldots,b_q)$
is a genuine $\tu K$-type, then its central character is $\chi_1$
iff
$a_i\in\bbZ+\frac12,\  b_j\in\bbZ$; 
its central character is $\chi_2$  iff  $a_i\in\bbZ,\  b_j\in\bbZ+\frac12.$
      
\end{itemize}

\end{lemma}

The following lemma is a slight modification of \cite[Lemma 6.8]{BTs18}. We include a proof for completeness.

\begin{lemma}\label{lem:genuine-infinitesimal-character-pattern}
Let $\widetilde G=\widetilde{\Spin}(c,d)$, with $c+d=2n$.
\begin{enumerate}
\item Suppose $c=2p$ and $d=2q$. Then the infinitesimal character of any genuine discrete series representation of $\widetilde G$ is conjugate to a parameter of the form
\begin{equation*}
(a_1,\ldots,a_p\mid b_1,\ldots,b_q),
\end{equation*}
where either
\begin{equation*}
a_i\in\mathbb Z,\quad b_j\in\mathbb Z+\frac12; \text{ or } \ 
a_i\in\mathbb Z+\frac12,\quad b_j\in\mathbb Z.
\end{equation*}

\item Suppose $c=2p+1$ and $d=2q+1$. Then the infinitesimal character of any genuine fundamental series representation of $\widetilde G$ is conjugate to a parameter of the form
\begin{equation*}
(a_1,\ldots,a_p\mid b_1,\ldots,b_q\mid x),
\end{equation*}
where either
\begin{equation*}
a_i\in\mathbb Z,\quad b_j\in\mathbb Z+\frac12; \text{ or } \ 
a_i\in\mathbb Z+\frac12,\quad b_j\in\mathbb Z.
\end{equation*}
 
If, in addition, the infinitesimal character is half-integral, then
$$x\in\mathbb Z  \ \text{ or }   \ x\in\mathbb Z+\frac12.$$

\end{enumerate}

\begin{proof}
We first consider the equal-rank case $c=2p$ and $d=2q$. Let $\pi$ be a genuine discrete series representation of $\widetilde G$, and let
\begin{equation*}
\lambda=(a_1,\ldots,a_p\mid b_1,\ldots,b_q)
\end{equation*}
be its Harish-Chandra parameter, written with respect to a compact Cartan subgroup.

Since $\pi$ is genuine, every $\widetilde K$-type occurring in $\pi$ is genuine. Let $\mu$ be the lowest $\widetilde K$-type of $\pi$. By the Blattner formula, the highest weight of $\mu$ differs from $\lambda$ by the fixed shift $\rho_n-\rho_c$, where $\rho_c$ is the half-sum of the positive compact roots and $\rho_n$ is the half-sum of the positive noncompact roots. In the present equal-rank case this shift has integral coordinates. Hence $\mu$ and $\lambda$ have the same integrality pattern.

On the other hand, by Remark \ref{r:hw}, a $\widetilde K$-type with highest weight
$(\mu_1,\ldots,\mu_p\mid \nu_1,\ldots,\nu_q)$ is genuine if and only if either
\begin{equation*}
\mu_i\in\mathbb Z,\quad \nu_j\in\mathbb Z+\frac12 ; \ \text{ or } \ 
\mu_i\in\mathbb Z+\frac12,\quad \nu_j\in\mathbb Z
\end{equation*}
for all $i,j$. Therefore the Harish-Chandra parameter $\lambda$ has the same mixed integrality pattern. This proves $(1)$.

Now suppose that $c=2p+1$ and $d=2q+1$. A fundamental series representation of $\widetilde{\Spin}(2p+1,2q+1)$ is induced from a cuspidal parabolic subgroup whose Levi factor has derived group locally isomorphic to $\Spin(2p,2q)$. The infinitesimal character may be written as
\begin{equation*}
(a_1,\ldots,a_p\mid b_1,\ldots,b_q\mid x),
\end{equation*}
where the first two blocks come from the discrete series parameter on this equal-rank Levi factor, and the last coordinate $x$ comes from the split part.

Since the representation is genuine, the inducing discrete series representation on the nonlinear cover of the Levi factor is genuine. Applying the result of $(1)$ to this equal-rank Levi factor, we conclude that either
\begin{equation*}
a_i\in\mathbb Z,\quad b_j\in\mathbb Z+\frac12;  \   \text{ or }
a_i\in\mathbb Z+\frac12,\quad b_j\in\mathbb Z
\end{equation*}
for all $i,j$. The remaining coordinate $x$ is the differential of the character on the split part. If the infinitesimal character is assumed to be half-integral, then necessarily $x\in\mathbb Z$ or $x\in\mathbb Z+\frac12$. This proves $(2)$.
\end{proof}

\end{lemma}

Now we specialize the groups to the ones under consideration and obtain the genuine central characters which are $\rho/2$-compatible.

\begin{prop}\label{prop:centers-compatible-central-characters}
Let $\widetilde G$ be one of the following nonlinear double covers of real spin groups. Then the center $Z(\widetilde G)$ and the number of genuine central characters are as follows:
\begin{equation*}
\begin{array}{c|c|c|c}
\widetilde G
& Z(\widetilde G)
& |\widehat Z_{\mathrm{gen}}(\widetilde G)|
& |\widehat Z_{\rho/2}(\widetilde G)| \\
\hline
\widetilde{\Spin}(2p+1,2p-1)
& \mathbb{Z}_2\times \mathbb{Z}_2
& 2
& 2 \\
\widetilde{\Spin}(2p+2,2p)
& \mathbb{Z}_2\times \mathbb{Z}_4
& 4
& 2 \\
\widetilde{\Spin}(2p+3,2p-1)
& \mathbb{Z}_2\times \mathbb{Z}_2
& 2
& 1
\end{array}
\end{equation*}

In particular, using the notation in Lemma \ref{l:central-char},
\begin{itemize}
    \item[(1)] For $\tu G=\widetilde{\Spin}(2p+1,2p-1)$,  $\widehat Z_{\rho/2}(\widetilde G)=\{\chi_1, \chi_2\}$.
    \item[(2)] For  $\tu G= \widetilde{\Spin}(2p+2,2p)$,
    $\widehat Z_{\rho/2}(\widetilde G)=\{\chi_2, \chi_4\}$.
    \item[(3)]  For $\tu G= \widetilde{\Spin}(2p+3,2p-1)$, 
        $\widehat Z_{\rho/2}(\widetilde G)=\{\chi_2\}$.
\end{itemize}

\begin{proof}
The description of the centers follows directly from Lemma \ref{l:center}. The number of genuine central characters follows from Lemma \ref{l:central-char}. Thus it remains to determine which genuine central characters are $\rho/2$-compatible.

Recall that in type $D_N$, with our choice of positive roots,
\begin{equation*}
\rho/2=
\left(\frac{N-1}{2},\frac{N-2}{2},\ldots,\frac12,0\right).
\end{equation*}
Thus, if $N=2r$, then $\rho/2$ has $r$ integral coordinates and $r$ half-integral coordinates. If $N=2r+1$, then $\rho/2$ has $r+1$ integral coordinates and $r$ half-integral coordinates. Since the Weyl group $W(D_N)$ acts by permutations and an even number of sign changes, the numbers of integral and half-integral coordinates are preserved under $W(D_N)$.

We first consider $\widetilde G=\widetilde{\Spin}(2p+1,2p-1)$. In this case the complex root system is of type $D_{2p}$, and hence $\rho/2$ has $p$ integral coordinates and $p$ half-integral coordinates. A genuine fundamental series parameter is conjugate to one of the form
\begin{equation*}
(a_1,\ldots,a_p\mid b_1,\ldots,b_{p-1}\mid x).
\end{equation*}
By Lemma \ref{lem:genuine-infinitesimal-character-pattern}, the two possible genuine integrality patterns are
\begin{equation*}
a_i\in\mathbb Z,\qquad b_j\in\mathbb Z+\frac12,
\end{equation*}
or
\begin{equation*}
a_i\in\mathbb Z+\frac12,\qquad b_j\in\mathbb Z.
\end{equation*}
The first pattern can be made $W(D_{2p})$-conjugate to $\rho/2$ by taking $x\in\mathbb Z+\frac12$, while the second can be made $W(D_{2p})$-conjugate to $\rho/2$   by taking $x\in\mathbb Z$. These two patterns correspond to the two genuine central characters $\chi_2$ and $\chi_1$, respectively, in the notation of Lemma \ref{l:central-char}. Hence both genuine central characters are $\rho/2$-compatible and we have  $\widehat Z_{\rho/2}(\widetilde G)=\{\chi_1, \chi_2\}$.

Next consider $\widetilde G=\widetilde{\Spin}(2p+2,2p)$. Here the complex root system is of type $D_{2p+1}$, so $\rho/2$ has $p+1$ integral coordinates and $p$ half-integral coordinates. Since $\widetilde G$ has a compact Cartan subgroup, a genuine discrete series parameter is conjugate to one of the form
\begin{equation*}
(a_1,\ldots,a_{p+1}\mid b_1,\ldots,b_p).
\end{equation*}
By Lemma \ref{lem:genuine-infinitesimal-character-pattern}, the possible genuine integrality patterns are
\begin{equation*}
a_i\in\mathbb Z,\qquad b_j\in\mathbb Z+\frac12,
\end{equation*}
or
\begin{equation*}
a_i\in\mathbb Z+\frac12,\qquad b_j\in\mathbb Z.
\end{equation*}
Only the first pattern can be $W(D_{2p+1})$-conjugate to $\rho/2$, because $\rho/2$ has only $p$ half-integral coordinates. The opposite pattern would require $p+1$ half-integral coordinates in the first block. Therefore the genuine central characters represented by the first pattern are precisely the $\rho/2$-compatible ones. Using the notation in Lemma \ref{l:central-char}, there are $\chi_2$ and $\chi_4$, and hence
 $\widehat Z_{\rho/2}(\widetilde G)=\{\chi_2, \chi_4\}$.

Finally, consider $\widetilde G=\widetilde{\Spin}(2p+3,2p-1)$. Again the complex root system is of type $D_{2p+1}$, so $\rho/2$ has $p+1$ integral coordinates and $p$ half-integral coordinates. A genuine fundamental series parameter is conjugate to one of the form
\begin{equation*}
(a_1,\ldots,a_{p+1}\mid b_1,\ldots,b_{p-1}\mid x).
\end{equation*}
The two possible genuine integrality patterns are
\begin{equation*}
a_i\in\mathbb Z,\qquad b_j\in\mathbb Z+\frac12,
\end{equation*}
or
\begin{equation*}
a_i\in\mathbb Z+\frac12,\qquad b_j\in\mathbb Z.
\end{equation*}
The first pattern is compatible with $\rho/2$, by taking $x\in\mathbb Z+\frac12$. This gives $p+1$ integral coordinates and $(p-1)+1=p$ half-integral coordinates. The second pattern is not compatible with $\rho/2$, since the first block alone would already contain $p+1$ half-integral coordinates, whereas $\rho/2$ has only $p$ half-integral coordinates. Hence only the central character corresponding to the first pattern is $\rho/2$-compatible. In the notation of Lemma \ref{l:central-char}, 
 $\widehat Z_{\rho/2}(\widetilde G)=\{\chi_2\}$.
\end{proof}

\end{prop}

We also introduce a notion of compatibility between a genuine central
character and the parameter data at the Lie algebra level.  Let
\[
\gamma^\circ=(\tu H,\overline\gamma)
\]
denote the parameter datum before choosing an irreducible genuine
representation \(\Gamma\) of \(\tu H\), and put
\[
d_{\gamma^\circ}
=
\overline\gamma
+\rho_i(\overline\gamma)
-2\rho_{i,c}(\overline\gamma).
\]

\begin{definition}\label{d:gamma-compatible}
A genuine central character
$\chi\in\widehat Z_{\mathrm{gen}}(\tu G)$
is said to be \(\gamma^\circ\)-compatible if there exists an
irreducible genuine representation \(\Gamma\) of \(\tu H\) such that
\[
d\Gamma=d_{\gamma^\circ}
\]
and \(\Gamma\) has central character \(\chi\) on \(Z(\tu G)\).

Equivalently, let \(\xi_{\gamma^\circ}\) be the character of
\(\tu H^0\) determined by \(d_{\gamma^\circ}\), and set
\[
C_H=Z(\tu G)\cap\tu H^0.
\]
Then
\[
\chi\text{ is \(\gamma^\circ\)-compatible}
\quad\Longleftrightarrow\quad
\chi|_{C_H}
=
\xi_{\gamma^\circ}|_{C_H}.
\]
\end{definition}

Indeed, since
\[
Z(\tu H)=Z(\tu G)\tu H^0
\]
by Proposition~\ref{prop:genuine-character-determined}, the condition
\[
\chi|_{C_H}
=
\xi_{\gamma^\circ}|_{C_H}
\]
is precisely what is needed for the characters \(\chi\) and
\(\xi_{\gamma^\circ}\) to define a character of \(Z(\tu H)\).  More
explicitly, one may define
\[
\zeta_{\chi,\gamma^\circ}(zh)
=
\chi(z)\xi_{\gamma^\circ}(h),
\qquad
z\in Z(\tu G),\quad h\in\tu H^0.
\]
The agreement on \(C_H\) guarantees that this character is
well-defined.  Proposition~\ref{prop:genuine-character-determined}
then gives the required irreducible genuine representation of
\(\tu H\).

We use the notation
\begin{equation} \label{e:bigchi}
\mathcal X(\gamma^\circ)
=
\left\{
\chi\in\widehat Z_{\mathrm{gen}}(\tu G)
\ \middle|\
\chi|_{C_H}
=
\xi_{\gamma^\circ}|_{C_H}
\right\}.
\end{equation}

Thus, if \(\gamma^\circ\) has infinitesimal character \(\rho/2\), then
\[
\mathcal X(\gamma^\circ)
\subseteq
\widehat Z_{\rho/2}(\tu G).
\]

\subsection{Regular characters for type $D$} \label{s:parameters}

In this subsection, we set up the notation for $\rho/2$-regular characters when fixing a $\rho/2$-compatible genuine central character for the nonlinear groups of type $D$ under consideration; more precisely, $\widetilde{\Spin}(n+1, n-1)$, with $n=2p$ or $2p+1$  and $\widetilde{\Spin}(n+2,n-2)$ with $n=2p+1$. 

We follow the notation of \cite[Section 6]{BTs18}. Let $\mathfrak g_0=\mathfrak{so}(c,d)$, where $c+d=2n$. Recall that the conjugacy classes of Cartan subalgebras of $\mathfrak g_0$ have representatives
\begin{equation*}
\mathfrak h^{r^+,r^-,m,s}
=
{(x_1^+,\ldots,x_{r^+}^+,
x_1^-,\ldots,x_{r^-}^-,
y_1,\ldots,y_m,y_{m+1},\ldots,y_{2m},
z_1,\ldots,z_s)},
\end{equation*}
with Cartan involution $\theta$ given by
\begin{equation*}
\theta(x_i^\pm)=x_i^\pm,\qquad
\theta(y_j)=y_{j+m},\qquad
\theta(z_k)=-z_k.
\end{equation*}
We denote the corresponding Cartan subgroup of $\widetilde G$ by $\widetilde H^{r^+,r^-,m,s}$.

Throughout the paper, we fix
\begin{equation*}
\Delta^+=\Delta^+(\rho/2)=\{e_i\pm e_j\mid i<j\},
\end{equation*}
and $\Delta ^+ = \Delta_1^+\cup \Delta_{1/2}^+$, where
\begin{align*}
\Delta_1^+
&=
R^+(\rho/2)
=\{e_i\pm e_j\mid i<j,\ i-j\in 2\mathbb Z\},\\
\Delta_{1/2}^+
&=
\{e_i\pm e_j\mid i<j,\ i-j\in 2\mathbb Z+1\}.
\end{align*}

\subsubsection{The group
\(\widetilde G=\widetilde{\Spin}(2p+1,2p-1)\).}

Put \(n=2p\). The Cartan subgroups of \(\widetilde G\) are of the form
\[
\widetilde H^{r+1,r,m,s},
\]
where
\[
s=2\ell-1,\qquad 1\leq \ell\leq p,
\qquad r,m\geq 0,
\qquad
2r+1+2m+s=n.
\]

Let
\[
\widetilde H_s:=\widetilde H^{1,0,0,n-1}
\]
be the most split Cartan subgroup. A Lie-level
\(\rho/2\)-regular parameter attached to \(\widetilde H_s\) is of the
form
\[
\gamma^\circ(\{k\}^*),\qquad 1\leq k\leq n.
\]
Here \(\gamma^\circ(\{k\}^*)\) means that the \(k\)-th coordinate
$\displaystyle\frac{n-k}{2}$
of
$\displaystyle\frac{\rho}{2}
=
\left(
\frac{n-1}{2},\frac{n-2}{2},\ldots,\frac12,0
\right)$
is placed on the compact part of \(\widetilde H_s\), while the
remaining coordinates are real.

The following lemma determines the genuine central characters
compatible with these Lie-level parameters.

\begin{lemma} \label{l:2p+12p-1-cenchar-comp}
Using the notation of Lemma~\ref{l:central-char} and
\eqref{e:bigchi}, we have
\[
\mathcal X\bigl(\gamma^\circ(\{i\}^*)\bigr)
=
\{\chi_1\},
\qquad
i\in\{1,3,\ldots,2p-1\},
\]
and
\[
\mathcal X\bigl(\gamma^\circ(\{j\}^*)\bigr)
=
\{\chi_2\},
\qquad
j\in\{2,4,\ldots,2p\}.
\]
\end{lemma}

\begin{proof}
Let
\[
C_s
=
Z(\widetilde G)\cap \widetilde H_s^0.
\]
The connected compact part of \(\widetilde H_s^0\) is a copy of
\(\Spin(2)\) contained in the \(\Spin(2p+1)\)-factor of the maximal
compact subgroup
\[
\widetilde K
=
\Spin(2p+1)\times\Spin(2p-1).
\]
Write
\[
Z(\widetilde G)
=
\{(\epsilon_+I,\epsilon_-I):\epsilon_\pm=\pm1\},
\]
and put
\[
z_+=(-I,I).
\]
The element \(z_+\) is the unique nontrivial central element contained
in this compact circle. Since the remaining connected part of
\(\widetilde H_s^0\) is split and hence contains no nontrivial
finite-order elements, we have
\[
C_s=\{1,z_+\}.
\]

There are no imaginary roots for the most split Cartan. Hence, for
\(\gamma^\circ(\{k\}^*)\),
\[
d_{\gamma^\circ(\{k\}^*)}
=
\overline\gamma_{\{k\}^*}.
\]
Let \(\xi_k\) be the character of \(\widetilde H_s^0\) determined by
this differential. Its weight on the connected compact circle is
$c_k=\displaystyle\frac{n-k}{2}.$

Consequently,
\[
\xi_k(z_+)
=
e^{2\pi i c_k}
=
e^{\pi i(n-k)}
=
(-1)^{n-k}.
\]
Since \(n=2p\) is even, this becomes
\[
\xi_k(z_+)=(-1)^k.
\]

By  \cite[Lemma 6.2]{BTs18},
\[
\chi_1(z_+)=-1,
\qquad
\chi_2(z_+)=1.
\]
A genuine central character \(\chi\) is
\(\gamma^\circ(\{k\}^*)\)-compatible precisely when
\[
\chi|_{C_s}=\xi_k|_{C_s}.
\]
It follows that \(\chi_1\) is compatible exactly when \(k\) is odd,
whereas \(\chi_2\) is compatible exactly when \(k\) is even.
\end{proof}

For \(1\leq k\leq n\), let
$\widetilde\gamma(\{k\}^*)$
denote the unique \(\rho/2\)-regular character whose Lie-level
parameter is \(\gamma^\circ(\{k\}^*)\).
Since Cayley transforms preserve the restriction of the Cartan
representation to \(Z(\widetilde G)\), every \(\rho/2\)-regular $\tu\gamma$
character of \(\widetilde G\) with central character \(\chi_1\) is
obtained from one of
\[
\widetilde\gamma(\{i\}^*),
\qquad
i\in\{1,3,\ldots,2p-1\},
\]
by a sequence of inverse Cayley transforms. Similarly, every
\(\rho/2\)-regular character  $\tu\gamma$  with central character \(\chi_2\) is
obtained from one of
\[
\widetilde\gamma(\{j\}^*),
\qquad
j\in\{2,4,\ldots,2p\}.
\]

Thus,
\[
\widetilde\gamma
=
c_{\{\alpha_1,\ldots,\alpha_l\}}
\bigl(\widetilde\gamma(\{k\}^*)\bigr),
\qquad 1\leq k\leq n,
\]
for some set
\(\{\alpha_1,\ldots,\alpha_l\}\subset\Delta_{1/2}^+\)
of strongly orthogonal roots.

If $\tu \gamma$ is attached to the Cartan subgroup $\widetilde H^{r+1,r,m,s}$, then it may be written in one of the following two forms:
\begin{equation} \label{e:param1}
\tu\gamma\bigl(
\{\epsilon_1 i_1,\epsilon_1 j_1\},
\{\epsilon_2 i_2,\epsilon_2 j_2\},
\ldots,
\{\epsilon_m i_m,\epsilon_m j_m\},
\{i_{m+1},\ldots,i_{m+r},i_{m+r+1},
j_{m+1},\ldots,j_{m+r}\}^*
\bigr),
\end{equation}
or
\begin{equation}\label{e:param2}
\tu\gamma\bigl(
\{\epsilon_1 i_1,\epsilon_1 j_1\},
\{\epsilon_2 i_2,\epsilon_2 j_2\},
\ldots,
\{\epsilon_m i_m,\epsilon_m j_m\},\{i_{m+1},\ldots,i_{m+r},
j_{m+1},\ldots,j_{m+r},j_{m+r+1}\}^*
\bigr),
\end{equation}
where $\epsilon_k=\pm 1$,
\begin{equation*}
i_k\in\{1,3,\ldots,2p-1\},
\qquad
j_k\in\{2,4,\ldots,2p\}.
\end{equation*}

This notation means that $e_{i_k}-\epsilon_k e_{j_k}$ is noncompact imaginary for $\gamma$ when $1\leq k\leq m$. Moreover, for the indices occurring in the starred block, the roots $e_{i_k}\pm e_{i_l}$ and $e_{j_k}\pm e_{j_l}$ are compact imaginary, while the roots $e_{i_k}\pm e_{j_l}$ are noncompact imaginary. We also note that the parameter in \eqref{e:param1} is obtained from $\tu\gamma(\{i\}^*)$ by a sequence of inverse Cayley transforms and the parameter in  \eqref{e:param2} is obtained from $\tu\gamma(\{j\}^*)$ by a sequence of inverse Cayley transforms.

\begin{remark}
    We note that
in \eqref{e:param1}, the starred block contains one more odd index
than even indices. Consequently, this parameter is obtained from
\(\widetilde\gamma(\{i\}^*)\) for some odd \(i\), and hence has
central character \(\chi_1\). Similarly, in \eqref{e:param2}, the starred block contains one more
even index than odd indices. It is therefore obtained from
\(\widetilde\gamma(\{j\}^*)\) for some even \(j\), and hence has
central character \(\chi_2\).
\end{remark}

\subsubsection{The group
\(\widetilde G=\widetilde{\Spin}(2p+2,2p)\).}

Put \(n=2p+1\). The conjugacy classes of Cartan subgroups of
\(\widetilde G\) are represented by
\[
\widetilde H^{r+1,r,m,s},
\]
where
\[
s=2\ell,\qquad 0\leq \ell\leq p,
\qquad r,m\geq0,
\qquad
2r+1+2m+s=n.
\]

Let
\[
\widetilde H_s:=\widetilde H^{1,0,0,2p}
\]
be the most split Cartan subgroup. The Lie-level
\(\rho/2\)-regular parameters attached to \(\widetilde H_s\) are
\[
\gamma^\circ(\{i\}^*),
\qquad
i\in\{1,3,\ldots,2p+1\}.
\]

Here \(\gamma^\circ(\{i\}^*)\) means that the \(i\)-th coordinate
$\displaystyle\frac{n-i}{2}$
of
$\rho/2$
is placed on the compact part of \(\widetilde H_s\), while the
remaining coordinates are real.

\begin{lemma}
For every \(i\in\{1,3,\ldots,2p+1\}\), we have
\[
\mathcal X\bigl(\gamma^\circ(\{i\}^*)\bigr)
=
\widehat Z_{\rho/2}(\widetilde G) =\{\chi_2, \chi_4\}.
\]
\end{lemma}
\begin{proof}
Let
\[
C_s
=
Z(\widetilde G)\cap \widetilde H_s^0.
\]

We can show that $C_s=\{1,z_+\}$, where $z_+=(-I, I)$ as in the proof of Lemma \ref{l:2p+12p-1-cenchar-comp}.

Let \(\xi_i\) be the character of \(\widetilde H_s^0\) determined by
$\overline\gamma_{\{ i\}^*}$. Its weight on the connected compact circle is
$c_i=\displaystyle\frac{n-i}{2}=\frac{2p+1-i}{2}$, which is always an integer since $i$ is odd. 

Consequently,
\[
\xi_i(z_+)
=
e^{2\pi i c_i}
=1. \]

On the other hand, 
by  \cite[Lemma 6.2]{BTs18},
\[
\chi_2(z_+)=\chi_4(z_+)=1,
\]
where $\chi_2$ and $\chi_4$ are the two elements in $\widehat{Z}_{\rho/2}(\tu G)$ by Proposition \ref{prop:centers-compatible-central-characters}. Hence,
$$
\chi_2|_{C_s} = \chi_4|_{C_s}=\xi_i|_{C_s}.
$$
The result follows.
\end{proof}

For \(i\in\{1,3,\ldots,2p+1\}\) and $\chi\in \widehat Z_{\rho/2}(\tu G)$,  let
$\widetilde\gamma(\{i\}^*)_\chi$
denote the  \(\rho/2\)-regular character whose Lie-level
parameter is \(\gamma^\circ(\{i\}^*)\) and genuine central character is $\chi$. When it is clear from the context, the subscript $\chi$ may be omitted. 
Since Cayley transforms preserve the restriction of the Cartan
representation to \(Z(\widetilde G)\), every \(\rho/2\)-regular 
character $\tu\gamma$ of \(\widetilde G\) with central character \(\chi\) is
obtained from one of
\[
\widetilde\gamma(\{i\}^*)_\chi,
\qquad
i\in\{1,3,\ldots,2p+1\},
\]
by a sequence of inverse Cayley transforms. 

Thus,
\[
\widetilde\gamma
=
c_{\{\alpha_1,\ldots,\alpha_l\}}
\bigl(\widetilde\gamma(\{i\}^*)_\chi \bigr),
\qquad  i\in\{1,3,\ldots,2p+1\},
\]
for some set
\(\{\alpha_1,\ldots,\alpha_l\}\subset\Delta_{1/2}^+\)
of strongly orthogonal roots.

If \(\widetilde\gamma\) is attached to the Cartan subgroup
\(\widetilde H^{r+1,r,m,s}\), then it may be written in the form
\[
\widetilde\gamma\bigl(
\{\epsilon_1 i_1,\epsilon_1 j_1\},
\{\epsilon_2 i_2,\epsilon_2 j_2\},
\ldots,
\{\epsilon_m i_m,\epsilon_m j_m\},
\{i_{m+1},\ldots,i_{m+r},i_{m+r+1},
j_{m+1},\ldots,j_{m+r}\}^*
\bigr)_\chi,
\]
where \(\epsilon_k=\pm1\),
\[
i_k\in\{1,3,\ldots,2p+1\},
\qquad
j_k\in\{2,4,\ldots,2p\}.
\]
This notation means that
$e_{i_k}-\epsilon_k e_{j_k}$
is noncompact imaginary for \(\widetilde\gamma\) when
\(1\leq k\leq m\). Moreover, for the indices occurring in the starred
block, the roots
$e_{i_k}\pm e_{i_l}$
and $e_{j_k}\pm e_{j_l}$
are compact imaginary, whereas the roots
$e_{i_k}\pm e_{j_l}$
are noncompact imaginary.

\begin{remark}
For each
\(\chi\in\widehat Z_{\rho/2}(\widetilde G)\), let
\(\mathcal B_{\rho/2,\chi}\) denote the block of genuine
\(\rho/2\)-regular characters with central character \(\chi\).
The preceding lemma shows that each Lie-level seed parameter
\[
\gamma^\circ(\{i\}^*),
\qquad
i\in\{1,3,\ldots,2p+1\},
\]
admits a genuine regular character in each of the two blocks
\[
\mathcal B_{\rho/2,\chi_2}
\qquad\text{and}\qquad
\mathcal B_{\rho/2,\chi_4}.
\]
Moreover, since cross actions through the integral Weyl group and Cayley transforms preserve the
genuine central character, the two blocks have the same
Lie-level parametrization described below.
\end{remark}

\subsubsection{The group
\(\widetilde G=\widetilde{\Spin}(2p+3,2p-1)\).}

Put \(n=2p+1\). The conjugacy classes of Cartan subgroups of
\(\widetilde G\) are represented by
\[
\widetilde H^{r+2,r,m,s},
\]
where
\[
s=2\ell-1,\qquad 1\leq\ell\leq p,
\qquad r,m\geq0,
\qquad
2r+2+2m+s=n.
\]

Let
\[
\widetilde H_s:=\widetilde H^{2,0,0,2p-1}
\]
be the most split Cartan subgroup. The Lie-level
\(\rho/2\)-regular parameters of the following type are
\[
\gamma^\circ(\{i_1,i_2\}^*),
\qquad
\{i_1,i_2\}\subset\{1,3,\ldots,2p+1\},
\quad i_1\neq i_2.
\]
Here \(\gamma^\circ(\{i_1,i_2\}^*)\) means that the
\(i_1\)-th and \(i_2\)-th coordinates
\[
\frac{n-i_1}{2},\qquad \frac{n-i_2}{2}
\]
of \(\rho/2\) are placed on the compact part of
\(\widetilde H_s\), while the remaining coordinates are real.

\begin{lemma}
For $\{i_1,i_2\}\subset\{1,3,\ldots,2p+1\},
\  i_1\neq i_2,$
we have
\[
\mathcal X\bigl(\gamma^\circ(\{i_1,i_2\}^*)\bigr)
=
\widehat Z_{\rho/2}(\widetilde G)
=
\{\chi_2\}.
\]
\end{lemma}

\begin{proof}
Let
\[
C_s
=
Z(\widetilde G)\cap \widetilde H_s^0.
\]
As in the proof of
Lemma~\ref{l:2p+12p-1-cenchar-comp}, we have
\[
C_s=\{1,z_+\},
\qquad
z_+=(-I,I).
\]

Let \(\xi_{i_1,i_2}\) be the character of
\(\widetilde H_s^0\) determined by
\(\overline\gamma_{\{i_1,i_2\}^*}\).
The two compact weights are
\[
c_{i_a}
=
\frac{n-i_a}{2}
=
\frac{2p+1-i_a}{2},
\qquad a=1,2.
\]
Since \(i_1\) and \(i_2\) are odd, both \(c_{i_1}\) and
\(c_{i_2}\) are integers. It follows that
\[
\xi_{i_1,i_2}(z_+)=1.
\]

By Proposition~\ref{prop:centers-compatible-central-characters},
$\widehat Z_{\rho/2}(\widetilde G)=\{\chi_2\},
$ and by \cite[Lemma 6.2]{BTs18},
$\chi_2(z_+)=1.$
Hence
\[
\chi_2|_{C_s}
=
\xi_{i_1,i_2}|_{C_s}.
\]
Therefore \(\chi_2\) is
\(\gamma^\circ(\{i_1,i_2\}^*)\)-compatible, and the result follows.
\end{proof}

For \(\{i_1,i_2\}\subset\{1,3,\ldots,2p+1\}\), $i_1\neq i_2$,  let
$\widetilde\gamma(\{i_1,i_2\}^*)$
denote the unique  \(\rho/2\)-regular character whose Lie-level
parameter is \(\gamma^\circ(\{i_1, i_2\}^*)\) and whose genuine central character is  $\chi_2$. 
These \(p(p+1)/2\) regular characters form a single orbit under
the cross action of \(W(\rho/2)\).
Every \(\rho/2\)-regular 
character $\tu\gamma$ of \(\widetilde G\) with central character \(\chi_2\) is
obtained from one of
\begin{equation*}
\tu\gamma(\{i_1,i_2\}^*),
\qquad
\{i_1,i_2\}\subset \{1,3,\ldots,2p+1\}, \ i_1\neq i_2.
\end{equation*}
by a sequence of inverse Cayley transforms. Thus, 
 \begin{equation*}
\tu\gamma=
c_{\{\alpha_1,\ldots,\alpha_l \}}\bigl(\tu\gamma( \{i_1,i_2\}^*)\bigr),
\end{equation*}
where $\{\alpha_1,\ldots,\alpha_l \}\subset \Delta_{1/2}^+$ is a set of strongly orthogonal roots.

If $\tu\gamma$ is attached to the Cartan subgroup $\widetilde H^{r+2,r,m,s}$, then it may be written in the form
\begin{equation*}
\tu\gamma\bigl(
\{\epsilon_1 i_1,\epsilon_1 j_1\},
\{\epsilon_2 i_2,\epsilon_2 j_2\},
\ldots,
\{\epsilon_m i_m,\epsilon_m j_m\},
\{i_{m+1},\ldots,i_{m+r+1},i_{m+r+2},
j_{m+1},\ldots,j_{m+r}\}^*
\bigr),
\end{equation*}
where $\epsilon_k=\pm 1$,
\begin{equation*}
i_k\in \{1,3,\ldots,2p+1\},
\qquad
j_k\in \{2,4,\ldots,2p\}.
\end{equation*}
This notation means that $e_{i_k}-\epsilon_k e_{j_k}$ is noncompact imaginary for $\gamma$ when $1\leq k\leq m$. For the indices occurring in the starred block, the roots $e_{i_k}\pm e_{i_l}$ and $e_{j_q}\pm e_{j_t}$ are compact imaginary, while the roots $e_{i_k}\pm e_{j_q}$ are noncompact imaginary.

\section{Small Representations}

In this section we study the small representations of the nonlinear real spin groups under consideration.

\subsection{Translation functors and $\tau$-invariants}

We recall the translation functors used to detect $\tau$-invariants. We state the definitions for $G$, but the same definitions apply to $\widetilde G$. For genuine representations of $\widetilde G$, translation functors are defined by tensoring with finite-dimensional representations of $G$ pulled back to $\widetilde G$. Since the kernel of $\widetilde G\to G$ acts trivially on such finite-dimensional representations, translation functors preserve genuineness.

Let $H\subset G$ be a Cartan subgroup, and let $\Lambda$ be the weight lattice coming from finite-dimensional representations of $G$. Let $\lambda\in\mathfrak h^*$ be a regular infinitesimal character. Fix a choice of positive integral roots for $\lambda$, and let $\Pi(\lambda)$ denote the corresponding set of simple integral roots.

For $\mu\in\Lambda$, let $F_\mu$ denote a finite-dimensional representation of $G$ whose set of extremal weights contains $\mu$.

\begin{definition}[{\cite[Definition 4.5.7]{V81}}]
Let $\pi$ be an admissible representation of $G$ with infinitesimal character $\lambda$. Suppose that $\lambda+\mu$ is also an infinitesimal character. The Jantzen--Zuckerman translation functor from $\lambda$ to $\lambda+\mu$ is defined by
\begin{equation*}
T_\lambda^{\lambda+\mu}(\pi)
:=
P_{\lambda+\mu}(\pi\otimes F_\mu),
\end{equation*}
where $P_{\lambda+\mu}$ denotes projection to the summand with generalized infinitesimal character $\lambda+\mu$.

For $\alpha\in\Pi(\lambda)$, choose an infinitesimal character
$\lambda_\alpha$ which is singular with respect to $\alpha$ and satisfies 
$\lambda-\lambda_\alpha\in\Lambda.$ We define
$$
\Psi _\alpha := T _\lambda^{\lambda_\alpha} 
$$
and call it the translation functor  to the $\alpha$-wall.

We also define the wall-crossing functor across the $\alpha$-wall by
\begin{equation*}
\psi_\alpha
:=T_{\lambda_\alpha}^{s_\alpha\lambda}\circ T_\lambda^{\lambda_\alpha}.
\end{equation*}
\end{definition}

We will also use translations between regular infinitesimal characters separated only by nonintegral walls. Namely, suppose $\lambda$ and $\lambda'$ are regular infinitesimal characters with $\lambda'-\lambda\in\Lambda$ and $R(\lambda)=R(\lambda')$. Then the translation functor $T_\lambda^{\lambda'}$ gives an equivalence between the corresponding regular blocks. Under the identification $\Pi(\lambda)=\Pi(\lambda')$, it preserves $\tau$-invariants. We will spell this out later. We refer to such a translation as translation across nonintegral walls.

\bigskip

Given a primitive ideal $I$ of the universal enveloping algebra $U(\mathfrak g)$, there is the Borho--Jantzen--Duflo $\tau$-invariant attached to $I$, denoted $\tau(I)$ (see \cite{V79a}). Since $G_{\mathbb C}$ is simply connected, there is an equivalent formulation of the $\tau$-invariant in terms of translation functors. We recall this definition below.

\begin{definition}
Let $\pi$ be an irreducible admissible representation of $G$ with regular infinitesimal character $\lambda$. The $\tau$-invariant of $\pi$ is the subset
\begin{equation*}
\tau(\pi)\subseteq \Pi(\lambda)
\end{equation*}
defined by
\begin{equation*}
\tau(\pi)
=
\{\alpha\in\Pi(\lambda)\mid \Psi_\alpha(\pi)=0\}.
\end{equation*}
We say that $\pi$ has maximal $\tau$-invariant if
\begin{equation*}
\tau(\pi)=\Pi(\lambda).
\end{equation*}
\end{definition}

The $\tau$-invariant is a measure of the size of an irreducible admissible representation: the larger the $\tau$-invariant, the smaller the representation.

\begin{definition}
Let $\pi$ be an irreducible genuine admissible representation of $\widetilde G$ with infinitesimal character $\rho/2$. We say that $\pi$ is small if it has maximal $\tau$-invariant.

We denote by
$\prod\nolimits_{\rho/2}^{s}(\widetilde G)$
the set of irreducible genuine small representations of $\widetilde G$ with infinitesimal character $\rho/2$. For a genuine central character $\chi$, we denote by
$\prod\nolimits_{\rho/2}^{s}(\widetilde G)_\chi$
the subset consisting of representations with central character $\chi$.
\end{definition}

Given a $\rho/2$-regular character $\gamma$, we will use the following criterion to compute the $\tau$-invariant of $J(\gamma)$.

\begin{thm}\label{t:tau}
\emph{(\cite[Theorem 4.12]{V79c} and \cite[Theorem 12.5]{H90})}
Suppose $\pi=J(\gamma)$, where
\begin{equation*}
\gamma=(\widetilde H,\Gamma,\overline\gamma)
\end{equation*}
is a $\lambda$-regular character. Let $\alpha\in\Pi(\lambda)$, and assume that $\alpha$ is simple with respect to $\Delta_\gamma^+$. Put
\begin{equation*}
l=
\frac{2\langle\alpha,\overline\gamma\rangle}
{\langle\alpha,\alpha\rangle}.
\end{equation*}
Then the following hold:
\begin{itemize}
\item[(a)] If $\alpha$ is real and $\gamma(m_\alpha)\neq (-1)^l\ep_\alpha$, then $\alpha\notin\tau(\pi)$.
\item[(b)] If $\alpha$ is real and $\gamma(m_\alpha)=(-1)^l\ep_\alpha$, then $\alpha\in\tau(\pi)$.
\item[(c)] If $\alpha$ is complex and $\theta\alpha\in\Delta_\gamma^+$, then $\alpha\notin\tau(\pi)$.
\item[(d)] If $\alpha$ is complex and $\theta\alpha\notin\Delta_\gamma^+$, then $\alpha\in\tau(\pi)$.
\item[(e)] If $\alpha$ is compact imaginary, then $\alpha\in\tau(\pi)$.
\item[(f)] If $\alpha$ is noncompact imaginary, then $\alpha\notin\tau(\pi)$.
\end{itemize}
\end{thm}

\begin{remark}
If $\alpha\in\Pi(\lambda)$ is simple with respect to $\Delta_\gamma^+$ and real for $\gamma$, then
\begin{equation*}
\alpha\notin\tau(J(\gamma)).
\end{equation*}
Indeed, by \cite[Lemma 6.8]{RT05}, the eigenvalues of $\gamma(m_\alpha)$ are $\pm i$, whereas
\begin{equation*}
\ep_\alpha(-1)^l=\pm 1
\end{equation*}
since $l=\langle\lambda,\alpha^\vee\rangle\in\bbZ$. Thus
\begin{equation*}
\gamma(m_\alpha)\neq \ep_\alpha(-1)^l,
\end{equation*}
and hence $\alpha\notin\tau(J(\gamma))$ by Theorem \ref{t:tau}(a).
\end{remark}

We will also need the following result for translation functors across nonintegral walls.

\begin{thm}\label{t:transfun}
\emph{(\cite[Theorem 5.3]{RT05})}
Let $\gamma$ be a genuine $\lambda$-regular character of $\widetilde G$. Suppose $\alpha$ is a nonintegral simple root in $\Delta^+(\overline\gamma)$. Let $\psi_\alpha$ denote the corresponding nonintegral wall-crossing functor. Then
\begin{equation*}
\psi_\alpha(J(\gamma))
=
\begin{cases}
J((s_\alpha\times\gamma)^\alpha)
& \text{if $\alpha$ is noncompact imaginary,} \\
J((s_\alpha\times\gamma)_\alpha)
& \text{if $\alpha$ is real satisfying the parity condition,} \\
J(s_\alpha\times\gamma)
& \text{otherwise.}
\end{cases}
\end{equation*}
\end{thm}

\begin{prop}\label{p:transtau}
Let $\gamma$ be a genuine $\lambda$-regular character of $\widetilde G$. Suppose $\alpha$ is a nonintegral simple root in $\Delta^+(\overline\gamma)$. Then the $\tau$-invariant is preserved by the nonintegral wall-crossing functor $\psi_\alpha$. More precisely, under the natural identification of the integral simple roots before and after crossing the nonintegral wall,
\begin{equation*}
\tau(J(\gamma))
=
\tau(\psi_\alpha(J(\gamma))).
\end{equation*}

\begin{proof}
Let $\lambda'=\lambda+\mu_\alpha$ be the infinitesimal character obtained from $\lambda$ by crossing the nonintegral $\alpha$-wall, as in \cite[Proposition 7.3.3]{V81}. Then the functor
\begin{equation*}
\psi_\alpha=T_\lambda^{\lambda'}
\end{equation*}
is an equivalence between the regular blocks with infinitesimal characters $\lambda$ and $\lambda'$.

Since $\lambda'-\lambda$ is integral, the integral root systems for $\lambda$ and $\lambda'$ are the same:
\begin{equation*}
R(\lambda)=R(\lambda').
\end{equation*}
Moreover, because $\alpha$ is nonintegral and simple, crossing the $\alpha$-wall does not change the positive integral roots. Hence we identify
\begin{equation*}
\Pi(\lambda)=\Pi(\lambda').
\end{equation*}

Let $\beta\in\Pi(\lambda)$. Choose $\lambda_\beta$ on the $\beta$-wall, and let $\lambda'_\beta$ be the corresponding infinitesimal character on the $\beta$-wall after crossing the nonintegral $\alpha$-wall. Translation functors commute, so we have a natural isomorphism
\begin{equation*}
T_{\lambda'}^{\lambda'_\beta}\circ T_\lambda^{\lambda'}
\simeq
T_{\lambda_\beta}^{\lambda'_\beta}\circ T_\lambda^{\lambda_\beta}.
\end{equation*}
In other words,
\begin{equation*}
\Psi'_\beta\circ \psi_\alpha
\simeq
\psi_{\alpha,\beta}\circ \Psi_\beta,
\end{equation*}
where $\Psi_\beta=T_\lambda^{\lambda_\beta}$, $\Psi'_\beta=T_{\lambda'}^{\lambda'_\beta}$, and $\psi_{\alpha,\beta}=T_{\lambda_\beta}^{\lambda'_\beta}$ is the corresponding nonintegral wall-crossing functor between the singular blocks.

By \cite[Proposition 7.3.3]{V81}, the functor $\psi_{\alpha,\beta}$ is also an equivalence. Therefore
\begin{equation*}
\Psi_\beta(J(\gamma))=0
\iff
\Psi'_\beta(\psi_\alpha(J(\gamma)))=0.
\end{equation*}
By the definition of the $\tau$-invariant, this is equivalent to
\begin{equation*}
\beta\in\tau(J(\gamma))
\iff
\beta\in\tau(\psi_\alpha(J(\gamma))).
\end{equation*}
Since this holds for every $\beta\in\Pi(\lambda)$, we obtain
\begin{equation*}
\tau(J(\gamma))=\tau(\psi_\alpha(J(\gamma)))
\end{equation*}
under the natural identification $\Pi(\lambda)=\Pi(\lambda')$.
\end{proof}

\end{prop}

\subsection{Small representations in type \(D\)}

For a \(\rho/2\)-regular character
\(\widetilde\gamma(\mathcal S)_\chi\) introduced in
Section~\ref{s:parameters}, we denote by
$J(\mathcal S)_\chi$
the corresponding irreducible quotient. 
When the genuine central
character is clear from the context, we omit the subscript \(\chi\)
and simply write \(J(\mathcal S)\).

We first introduce a set of
representations that will be shown to be small.

\begin{definition}\label{d:list}
Let \(\chi\) be a \(\rho/2\)-compatible genuine central character.
We define
\(\mathcal R_{\rho/2}(\widetilde G)_\chi\)
as follows.

\begin{itemize}
\item[(a)]
If $\widetilde G=\widetilde{\Spin}(2p+1,2p-1),$
then
$$
\mathcal R_{\rho/2}(\widetilde G)_{\chi}
=\begin{cases}
    \left\{
J(\{1,2\},\{3,4\},\ldots,
  \{2p-3,2p-2\},\{2p-1\}^*)_{\chi_1}
\right\}, &\text{ if }  \chi=\chi_1,     \\
\left\{
J(\{2p\}^*)_{\chi_2}
\right\}, &\text{ if  } \chi=\chi_2.
\end{cases}
$$
\item[(b)]
If $\widetilde G=\widetilde{\Spin}(2p+2,2p),$
then
\begin{align*}
\mathcal R_{\rho/2}(\widetilde G)_\chi
=\bigl\{
&
J(\{2p+1\}^*)_\chi,\\
&
J(\{1,2\},\{3,4\},\ldots,
  \{2p-1,2p\},\{2p+1\}^*)_\chi,\\
&
J(\{1,2\},\{3,4\},\ldots,
  \{2p-3,2p-2\},
  \{-(2p-1),-2p\},\{2p+1\}^*)_\chi
\bigr\}.
\end{align*}

\item[(c)]
If $\widetilde G=\widetilde{\Spin}(2p+3,2p-1),$ then
\[
\mathcal R_{\rho/2}(\widetilde G)_\chi
=
\left\{
J(\{1,2\},\{3,4\},\ldots,
  \{2p-3,2p-2\},
  \{2p-1,2p+1\}^*)_\chi
\right\}.
\]

\end{itemize}
\end{definition}
In cases (a) and (c), the genuine central character is uniquely
determined by the parameter, so we will often omit the subscript
\(\chi\) from \(J(\mathcal S)_\chi\).

\begin{prop}\label{p:RD1}
Let \(\widetilde G\) be one of the groups in
Definition~\ref{d:list}. For each \(\rho/2\)-compatible genuine
central character \(\chi\), we have
\[
\mathcal R_{\rho/2}(\widetilde G)_\chi
\subseteq
\prod\nolimits_{\rho/2}^{s}(\widetilde G)_\chi.
\]
\end{prop}

\begin{proof}

It suffices to prove
\[
\Pi(\rho/2)\subseteq\tau(J(\gamma))
\]
for each \(J(\gamma)\) listed in Definition~\ref{d:list}. 

Throughout, we assume that $p\ge 2$. For case (b) with $p=1$, the proof is similar. 

Recall that
\[
\Pi(\rho/2)=
\begin{cases}
\{e_i-e_{i+2}:1\leq i\leq2p-2\}
\cup
\{e_{2p-3}+e_{2p-1},
  e_{2p-2}+e_{2p}\},
&\text{in case (a)},\\[1mm]
\{e_i-e_{i+2}:1\leq i\leq2p-1\}
\cup
\{e_{2p-1}+e_{2p+1},
  e_{2p-2}+e_{2p}\},
&\text{in cases (b) and (c)}.
\end{cases}
\]

We repeatedly use the following argument. Suppose that
\[
\psi_{\beta_l}\cdots\psi_{\beta_1}(J(\gamma))
=
J(\delta),
\qquad
w=s_{\beta_l}\cdots s_{\beta_1},
\]
and that \(\alpha\in\Pi(\rho/2)\) is simple with respect to
\(w(\Delta^+)\). If \(\alpha\) is compact imaginary for \(\delta\),
or if it is complex and
$\theta_\delta(\alpha)\notin w(\Delta^+),$ then Theorem~\ref{t:tau} gives
\(\alpha\in\tau(J(\delta))\), and
Proposition~\ref{p:transtau} implies
\[
\alpha\in\tau(J(\gamma)).
\]

We now treat the three families separately.
\medskip

For (a), first consider
$\gamma=
\tu\gamma(\{1,2\},\{3,4\},\ldots,
\{2p-3,2p-2\},\{2p-1\}^*).$
Its Cartan involution is
\[
\theta_\gamma(x_1,\ldots,x_{2p})
=
(-x_2,-x_1,-x_4,-x_3,\ldots,
-x_{2p-2},-x_{2p-3},x_{2p-1},-x_{2p}).
\]

We first consider the roots \(e_i-e_{i+2}\).
Let \(i\) be odd with \(1\leq i\leq 2p-5\), and put
$\beta=e_{i+1}-e_{i+2}.$
Then
$e_i-e_{i+2}$ and 
$e_{i+1}-e_{i+3}$
are simple with respect to \(s_\beta(\Delta^+)\).
Since \(\beta\) is complex for \(\gamma\), Theorem~\ref{t:transfun}
gives
\[
\psi_\beta(J(\gamma))
=
J(\delta),
\qquad
\delta:=s_\beta\times\gamma.
\]
Moreover, \(\theta_\delta=\theta_\gamma\), and both roots above are
complex for \(\delta\). We have
\begin{align*}
\theta_\delta(e_i-e_{i+2})
&=-e_{i+1}+e_{i+3}
   \notin s_\beta(\Delta^+),\\
\theta_\delta(e_{i+1}-e_{i+3})
&=-e_i+e_{i+2}
   \notin s_\beta(\Delta^+).
\end{align*}
Hence the preceding observation shows that
\[
e_i-e_{i+2}\in\tau(J(\gamma)),
\qquad
1\leq i\leq 2p-4.
\]

For the remaining roots near the end of the Dynkin diagram, take
$\beta=e_{2p-2}-e_{2p-1}.$
Then
$e_{2p-3}-e_{2p-1},
e_{2p-2}-e_{2p},$
and 
$e_{2p-2}+e_{2p}$
are simple with respect to \(s_\beta(\Delta^+)\).
Again \(\beta\) is complex for \(\gamma\), so for
$\delta=s_\beta\times\gamma$, we have \(\theta_\delta=\theta_\gamma\), and all three roots are
complex for \(\delta\). Furthermore,
\begin{align*}
\theta_\delta(e_{2p-3}-e_{2p-1})
&=-e_{2p-2}+e_{2p-1}
  \notin s_\beta(\Delta^+),\\
\theta_\delta(e_{2p-2}-e_{2p})
&=-e_{2p-3}+e_{2p}
  \notin s_\beta(\Delta^+),\\
\theta_\delta(e_{2p-2}+e_{2p})
&=-e_{2p-3}-e_{2p}
  \notin s_\beta(\Delta^+).
\end{align*}
Thus
\[
e_{2p-3}-e_{2p-1},
\qquad
e_{2p-2}\pm e_{2p}
\]
belong to \(\tau(J(\gamma))\).

It remains to consider \(e_{2p-3}+e_{2p-1}\). Set
\[
w=s_{\beta_3}s_{\beta_2}s_{\beta_1},
\]
where
\[
\beta_1=e_{2p-3}-e_{2p-2},
\qquad
\beta_2=e_{2p-1}-e_{2p},
\qquad
\beta_3=e_{2p-1}+e_{2p}.
\]
Then \(e_{2p-3}+e_{2p-1}\) is simple with respect to
\(w(\Delta^+)\). According to the root types encountered along the
successive wall crossings, define
\begin{align*}
\delta_1
&:=(s_{\beta_1}\times\gamma)^{\beta_1},
&&\text{since \(\beta_1\) is noncompact imaginary for \(\gamma\)},\\
\delta_2
&:=s_{\beta_2}\times\delta_1,
&&\text{since \(\beta_2\) is complex for \(\delta_1\)},\\
\delta=\delta_3
&:=s_{\beta_3}\times\delta_2,
&&\text{since \(\beta_3\) is complex for \(\delta_2\)}.
\end{align*}
By Theorem~\ref{t:transfun},
\[
\psi_{\beta_3}\psi_{\beta_2}\psi_{\beta_1}(J(\gamma))
=
J(\delta).
\]
The Cartan involution of \(\delta\) is
\[
\theta_\delta(x_1,\ldots,x_{2p})
=
(-x_2,-x_1,\ldots,
-x_{2p-4},-x_{2p-5},
-x_{2p-3},-x_{2p-2},
x_{2p-1},-x_{2p}).
\]
Hence \(e_{2p-3}+e_{2p-1}\) is complex for \(\delta\), and
\[
\theta_\delta(e_{2p-3}+e_{2p-1})
=
-e_{2p-3}+e_{2p-1}
\notin w(\Delta^+).
\]
Therefore
\[
e_{2p-3}+e_{2p-1}\in\tau(J(\gamma)).
\]
It follows that \(J(\gamma)\) has maximal \(\tau\)-invariant.

Now consider
$\gamma=\gamma(\{2p\}^*).$
In this case
\[
\theta_\gamma(x_1,\ldots,x_{2p})
=
(-x_1,\ldots,-x_{2p-1},x_{2p}).
\]

Let \(i\) be odd with \(1\leq i\leq 2p-3\), and set
$\beta=e_{i+1}-e_{i+2}.$
The roots
$e_i-e_{i+2}$ and
$e_{i+1}-e_{i+3}$
are simple with respect to \(s_\beta(\Delta^+)\).
Since \(\beta\) is real for \(\gamma\), Theorem~\ref{t:transfun}
gives
\[
\psi_\beta(J(\gamma))
=
J(\delta),
\qquad
\delta:=(s_\beta\times\gamma)_\beta.
\]
The root \(\beta\) is noncompact imaginary for \(\delta\), and
\[
\theta_\delta(x_1,\ldots,x_{2p})
=
(-x_1,\ldots,-x_i,-x_{i+2},-x_{i+1},
-x_{i+3},\ldots,-x_{2p-1},x_{2p}).
\]
Thus the two roots above are complex for \(\delta\), and
\begin{align*}
\theta_\delta(e_i-e_{i+2})
&=-e_i+e_{i+1}
  \notin s_\beta(\Delta^+),\\
\theta_\delta(e_{i+1}-e_{i+3})
&=-e_{i+2}+e_{i+3}
  \notin s_\beta(\Delta^+).
\end{align*}
It follows that
\[
e_i-e_{i+2}\in\tau(J(\gamma)),
\qquad
1\leq i\leq2p-2.
\]

For \(i=2p-3\), the same wall crossing is given by
$\beta=e_{2p-2}-e_{2p-1}.$
Besides the two roots considered above,
\(e_{2p-2}+e_{2p}\) is also simple with respect to
\(s_\beta(\Delta^+)\). For
$\delta=(s_\beta\times\gamma)_\beta$, 
we have
\[
\theta_\delta(x_1,\ldots,x_{2p})
=
(-x_1,\ldots,-x_{2p-3},
-x_{2p-1},-x_{2p-2},x_{2p}),
\]
and hence \(e_{2p-2}+e_{2p}\) is complex for \(\delta\), with
\[
\theta_\delta(e_{2p-2}+e_{2p})
=
-e_{2p-1}+e_{2p}
\notin s_\beta(\Delta^+).
\]
Therefore
\[
e_{2p-2}+e_{2p}\in\tau(J(\gamma)).
\]

It remains to consider \(e_{2p-3}+e_{2p-1}\). Let
\[
w=s_{\beta_3}s_{\beta_2}s_{\beta_1},
\]
where
\[
\beta_1=e_{2p-3}-e_{2p-2},
\qquad
\beta_2=e_{2p-1}-e_{2p},
\qquad
\beta_3=e_{2p-1}+e_{2p}.
\]
Then \(e_{2p-3}+e_{2p-1}\) is simple with respect to
\(w(\Delta^+)\). Define successively
\begin{align*}
\delta_1
&:=(s_{\beta_1}\times\gamma)_{\beta_1},
&&\text{since \(\beta_1\) is real for \(\gamma\)},\\
\delta_2
&:=s_{\beta_2}\times\delta_1,
&&\text{since \(\beta_2\) is complex for \(\delta_1\)},\\
\delta=\delta_3
&:=s_{\beta_3}\times\delta_2,
&&\text{since \(\beta_3\) is complex for \(\delta_2\)}.
\end{align*}
By Theorem~\ref{t:transfun},
\[
\psi_{\beta_3}\psi_{\beta_2}\psi_{\beta_1}(J(\gamma))
=
J(\delta).
\]
The Cartan involution of \(\delta\) is
\[
\theta_\delta(x_1,\ldots,x_{2p})
=
(-x_1,\ldots,-x_{2p-4},
-x_{2p-2},-x_{2p-3},
-x_{2p-1},x_{2p}).
\]
Therefore \(e_{2p-3}+e_{2p-1}\) is complex for \(\delta\), and
\[
\theta_\delta(e_{2p-3}+e_{2p-1})
=
-e_{2p-2}-e_{2p-1}
\notin w(\Delta^+).
\]
Hence
\[
e_{2p-3}+e_{2p-1}\in\tau(J(\gamma)).
\]
Thus \(J(\{2p\}^*)\) also has maximal \(\tau\)-invariant.

For (b), we suppress the subscript \(\chi\), since the following
calculations are the same for each
\(\chi\in\widehat Z_{\rho/2}(\widetilde G)\).

\begin{itemize}

\item
The proof for \(J(\{2p+1\}^*)\) is identical to that for
\(J(\{2p\}^*)\) in (a), with \(2p\) replaced by \(2p+1\).

\item
Let $\gamma=
\tu\gamma(\{1,2\},\{3,4\},\ldots,
\{2p-1,2p\},\{2p+1\}^*).$
Then
\[
\theta_\gamma(x_1,\ldots,x_{2p+1})
=
(-x_2,-x_1,-x_4,-x_3,\ldots,
-x_{2p},-x_{2p-1},x_{2p+1}).
\]

For odd \(i\), \(1\leq i\leq2p-3\), take
$\beta=e_{i+1}-e_{i+2}.$
As in (a), \(\beta\) is complex for \(\gamma\), and after the cross
action \(s_\beta\times\gamma\), the roots
$e_i-e_{i+2}$ and $e_{i+1}-e_{i+3}$ are complex descents. Hence
\[
e_i-e_{i+2}\in\tau(J(\gamma)),
\qquad 1\leq i\leq2p-2.
\]

Next take
$\beta=e_{2p-1}-e_{2p}.$
This root is noncompact imaginary for \(\gamma\). For $\delta=(s_\beta\times\gamma)^\beta,$ the roots \(e_{2p-1}\pm e_{2p+1}\) are complex, and
\begin{align*}
\theta_\delta(e_{2p-1}-e_{2p+1})
&=-e_{2p-1}-e_{2p+1}
  \notin s_\beta(\Delta^+),\\
\theta_\delta(e_{2p-1}+e_{2p+1})
&=-e_{2p-1}+e_{2p+1}
  \notin s_\beta(\Delta^+).
\end{align*}
Thus
\[
e_{2p-1}\pm e_{2p+1}\in\tau(J(\gamma)).
\]

It remains to consider \(e_{2p-2}+e_{2p}\). Set
\[
w=s_{\beta_3}s_{\beta_2}s_{\beta_1},
\]
where $\beta_1=e_{2p-2}-e_{2p-1},\
\beta_2=e_{2p}-e_{2p+1},\ 
\beta_3=e_{2p}+e_{2p+1}.$

All three roots are complex at the corresponding stages, so
\[
\delta
=
s_{\beta_3}\times
\bigl(s_{\beta_2}\times(s_{\beta_1}\times\gamma)\bigr)
\]
satisfies
\[
\psi_{\beta_3}\psi_{\beta_2}\psi_{\beta_1}(J(\gamma))
=
J(\delta),
\qquad
\theta_\delta=\theta_\gamma.
\]
Since \(e_{2p-2}+e_{2p}\) is simple with respect to \(w(\Delta^+)\)
and
\[
\theta_\delta(e_{2p-2}+e_{2p})
=
-e_{2p-3}-e_{2p-1}
\notin w(\Delta^+),
\]
we obtain
\[
e_{2p-2}+e_{2p}\in\tau(J(\gamma)).
\]
Hence \(J(\gamma)\) has maximal \(\tau\)-invariant.

\item
Let $\gamma=
\tu \gamma(\{1,2\},\{3,4\},\ldots,
\{-(2p-1),-2p\},\{2p+1\}^*).$
Then
\[
\theta_\gamma(x_1,\ldots,x_{2p+1})
=
(-x_2,-x_1,-x_4,-x_3,\ldots,
-x_{2p-2},-x_{2p-3},
x_{2p},x_{2p-1},x_{2p+1}).
\]

For odd \(i\), \(1\leq i\leq2p-5\), the same argument as above,
with
$\beta=e_{i+1}-e_{i+2},$
gives
\[
e_i-e_{i+2}\in\tau(J(\gamma)),
\qquad
1\leq i\leq2p-4.
\]

Taking
$\beta=e_{2p-2}-e_{2p-1},$
which is complex for \(\gamma\), gives in addition
\[
e_{2p-3}-e_{2p-1},
\quad
e_{2p-2}-e_{2p}
\in\tau(J(\gamma)).
\]

Next, \(\beta=e_{2p-1}-e_{2p}\) is real for \(\gamma\). For
$\delta=(s_\beta\times\gamma)_\beta,$ the roots \(e_{2p-1}\pm e_{2p+1}\) are compact imaginary.
Therefore
\[
e_{2p-1}\pm e_{2p+1}\in\tau(J(\gamma)).
\]

Finally, using the same three-step translation as above,
\[
\beta_1=e_{2p-2}-e_{2p-1},\qquad
\beta_2=e_{2p}-e_{2p+1},\qquad
\beta_3=e_{2p}+e_{2p+1},
\]
we obtain a parameter \(\delta\) for which
\(e_{2p-2}+e_{2p}\) is a complex descent. Indeed,
\[
\theta_\delta(e_{2p-2}+e_{2p})
=
-e_{2p-3}+e_{2p-1}
\notin
s_{\beta_3}s_{\beta_2}s_{\beta_1}(\Delta^+).
\]
Hence
\[
e_{2p-2}+e_{2p}\in\tau(J(\gamma)).
\]
Thus this representation also has maximal \(\tau\)-invariant.

\end{itemize}

\medskip

For (c), let $\gamma=
\tu\gamma(\{1,2\},\{3,4\},\ldots,
\{2p-3,2p-2\},\{2p-1,2p+1\}^*).$
Then
\[
\theta_\gamma(x_1,\ldots,x_{2p+1})
=
(-x_2,-x_1,-x_4,-x_3,\ldots,
-x_{2p-2},-x_{2p-3},
x_{2p-1},-x_{2p},x_{2p+1}).
\]

For odd \(i\) with \(1\leq i\leq 2p-5\), take
$\beta=e_{i+1}-e_{i+2}.$
As in case (a), \(\beta\) is complex for \(\gamma\), and the same
single-wall argument gives
\[
e_i-e_{i+2},\ e_{i+1}-e_{i+3}
\in\tau(J(\gamma)).
\]
Hence
\[
e_i-e_{i+2}\in\tau(J(\gamma)),
\qquad 1\leq i\leq 2p-4.
\]

Next take
$\beta=e_{2p-2}-e_{2p-1}.$
Again \(\beta\) is complex for \(\gamma\). For
$\delta=s_\beta\times\gamma,$
the roots
$e_{2p-3}-e_{2p-1}$ and
$e_{2p-2}\pm e_{2p}$
are simple with respect to \(s_\beta(\Delta^+)\) and are complex
descents for \(\delta\), by the same calculation as in case (a).
Therefore
\[
e_{2p-3}-e_{2p-1},
\qquad
e_{2p-2}\pm e_{2p}
\in\tau(J(\gamma)).
\]

Finally, since \(2p-1\) and \(2p+1\) occur in the starred block,
the root
$e_{2p-1}+e_{2p+1}$
is compact imaginary for \(\gamma\). Hence
\[
e_{2p-1}+e_{2p+1}\in\tau(J(\gamma)).
\]
Thus
$\Pi(\rho/2)\subseteq\tau(J(\gamma)),$
and \(J(\gamma)\) has maximal \(\tau\)-invariant.
\end{proof}

\begin{remark}\label{r:number}
Let
\begin{equation*}
\mathcal{R}_{\rho/2}(\tu G)
=
\bigsqcup_{\chi\in \widehat Z_{\rho/2}(\tu G)}
\mathcal{R}_{\rho/2}(\tu G)_\chi.
\end{equation*}
By Definition~\ref{d:list} and
Proposition~\ref{prop:centers-compatible-central-characters},
we obtain the following cardinalities:

\begin{equation*}
\begin{array}{c|c|c|c}
\tu G
&
|\widehat Z_{\rho/2}(\tu G)|
&
|\mathcal{R}_{\rho/2}(\tu G)_\chi|
&
|\mathcal{R}_{\rho/2}(\tu G)|
\\
\hline
\tu{\Spin}(2p+1,2p-1)
& 2 & 1 & 2
\\
\tu{\Spin}(2p+2,2p)
& 2 & 3 & 6
\\
\tu{\Spin}(2p+3,2p-1)
& 1 & 1 & 1
\end{array}
\end{equation*}
Here $\chi$ denotes a fixed $\rho/2$-compatible genuine central character.
\end{remark}

We conclude this section with the following exhaustion result, whose
proof will be given in the next section.

\begin{thm}\label{t:small-exhaustion}
Let
\begin{equation*}
\widetilde G\in
\{\widetilde{\Spin}(2p+1,2p-1),
\widetilde{\Spin}(2p+2,2p),
\widetilde{\Spin}(2p+3,2p-1)\}.
\end{equation*}
For every $\rho/2$-compatible genuine central character $\chi$, we have
\begin{equation*}
\mathcal R_{\rho/2}(\widetilde G)_\chi
=
\prod\nolimits_{\rho/2}^{s}(\widetilde G)_\chi.
\end{equation*}
Consequently,
\begin{equation*}
\mathcal R_{\rho/2}(\widetilde G)
=
\prod\nolimits_{\rho/2}^{s}(\widetilde G).
\end{equation*}
\end{thm}

To prove Theorem~\ref{t:small-exhaustion}, we count the small
representations of \(\widetilde G\) with fixed genuine central
character using the coherent continuation action of the integral
Weyl group. The resulting cardinalities agree with those in
Remark~\ref{r:number}, and Proposition~\ref{p:RD1} then yields the
desired equalities. This will be carried out in the next section.

\section{Counting Small Representations}
\label{s:counting}

\subsection{Coherent continuation}

We recall the coherent-continuation representations used to count small representations. Given an infinitesimal character $\lambda$, let $\calF(\lambda)$ be the family of infinitesimal characters obtained from $\lambda$ by crossing nonintegral walls. For $\lambda'\in\calF(\lambda)$ and a fixed genuine central character $\chi$, let $\calB_{\lambda',\chi}$ denote the set of equivalence classes of standard representation parameters with infinitesimal character $\lambda'$ and central character $\chi$.

For the counting argument, we fix $\lambda=\rho/2$ and a $\rho/2$-compatible genuine central character $\chi$. The integral Weyl group $W(\lambda)$ acts on $\calB_{\lambda,\chi}$ by the cross action. It also acts on $\bbZ[\calB_{\lambda,\chi}]$ by coherent continuation.

Let $w\in W(\lambda)$, and write
\begin{equation}\label{e:sim-refl}
w=s_{\beta_\ell}\cdots s_{\beta_1},
\end{equation}
where each $\beta_k$ is simple for the chamber
$s_{\beta_{k-1}}\cdots s_{\beta_1}(\Delta^+)$.
For $\gamma\in\calB_{\lambda,\chi}$, let $m(\gamma,w)$ be the number of imaginary roots among the $\beta_k$, computed successively along the cross action. Then the coherent continuation formula can be expressed as: 
\begin{equation}\label{e:cohcon}
w\cdot \gamma
=
(-1)^{m(\gamma,w)}\,w\times\gamma
+
\text{(terms attached to more split Cartan subgroups)}.
\end{equation}

Choose representatives $\gamma_j$ for the cross-action orbits of $W(\lambda)$ on $\calB_{\lambda,\chi}$, and let
\begin{equation*}
W_{\gamma_j}
=
\{w\in W(\lambda)\mid w\times\gamma_j=\gamma_j\}.
\end{equation*}
Then as $W(\lambda)$-representations, 
\begin{equation}\label{e:ind}
\bbZ[\calB_{\lambda,\chi}]
\simeq
\bigoplus_j
\operatorname{Ind}_{W_{\gamma_j}}^{W(\lambda)}(\epsilon_j),
\end{equation}
where $\epsilon_j$ is the one-dimensional representation of $W_{\gamma_j}$ defined by
\begin{equation*}
w\cdot\gamma_j
=
\epsilon_j(w)\gamma_j
+
\text{terms attached to more split Cartan subgroups}.
\end{equation*}

By Frobenius reciprocity,
\begin{equation*}
[\operatorname{sgn}_{W(\lambda)}:\bbZ[\calB_{\lambda,\chi}]]
=
\sum_j
[
\operatorname{sgn}_{W(\lambda)}|_{W_{\gamma_j}}:\epsilon_j
].
\end{equation*}
Thus the number of small representations with central character $\chi$ is the number of representatives $\gamma_j$ satisfying
\begin{equation}\label{e:key}
\operatorname{sgn}_{W(\lambda)}|_{W_{\gamma_j}}
=
\epsilon_j;
\end{equation}
See \cite{Ts19} for more details. 
\medskip

We will use the following description of the cross stabilizer. Let $\gamma\in\calB_{\lambda,\chi}$. Write
\begin{equation*}
R(\lambda)=R^i_\gamma(\lambda)\sqcup R^r_\gamma(\lambda)\sqcup R^C_\gamma(\lambda),
\end{equation*}
where $R^i_\gamma(\lambda)$, $R^r_\gamma(\lambda)$, and $R^C_\gamma(\lambda)$ are the integral imaginary, real, and complex roots for $\gamma$, respectively. Let
\begin{equation*}
W^i_\gamma(\lambda)=W(R^i_\gamma(\lambda)),
\qquad
W^r_\gamma(\lambda)=W(R^r_\gamma(\lambda)).
\end{equation*}
Furthermore, let $W^C_\gamma(\lambda)^\theta$ denote the subgroup generated by elements of the form $s_\alpha s_{\theta\alpha}$,
where $\alpha\in R^C_\gamma(\lambda)$ and $s_\alpha s_{\theta\alpha}$ preserves the regular character $\gamma$. Then the cross stabilizer of $\gamma$ is
\begin{equation}\label{e:stab}
W_\gamma=
\left[
W^i_\gamma(\lambda)\times W^r_\gamma(\lambda)
\right]
\rtimes
W^C_\gamma(\lambda)^\theta.
\end{equation}

\bigskip
\subsection{Type \(D_n\)}

Let
\[
\widetilde G\in
\{
\widetilde{\Spin}(2p+1,2p-1),
\widetilde{\Spin}(2p+2,2p),
\widetilde{\Spin}(2p+3,2p-1)
\}.
\]
The Cartan subgroups of the groups under consideration are as follows:
\begin{align*}
&\hspace*{1cm}
\widetilde G
&&
\widetilde H=\widetilde H^{r_+,r_-,m,s}
\\
&\widetilde{\Spin}(2p+1,2p-1)
&&
\widetilde H^{r,r-1,m,s},
&
&s=2k-1,\quad 1\leq k\leq p,
\\
&&&&&
1\leq r\leq p+\frac{1-s}{2},
\qquad
m=p-r+\frac{1-s}{2},
\\[1mm]
&\widetilde{\Spin}(2p+2,2p)
&&
\widetilde H^{r,r-1,m,s},
&
&s=2k,\quad 0\leq k\leq p,
\\
&&&&&
1\leq r\leq p+1-\frac{s}{2},
\qquad
m=p-r+1-\frac{s}{2},
\\[1mm]
&\widetilde{\Spin}(2p+3,2p-1)
&&
\widetilde H^{r,r-2,m,s},
&
&s=2k-1,\quad 1\leq k\leq p,
\\
&&&&&
2\leq r\leq p+\frac{3-s}{2},
\qquad
m=p-r+\frac{3-s}{2}.
\end{align*}

\begin{remark}\label{r:orbit}
For each group under consideration, after fixing a
\(\rho/2\)-compatible genuine central character \(\chi\), the
\(\rho/2\)-regular characters attached to a fixed Cartan subgroup
form a single cross-action orbit of \(W(\rho/2)\).

There is a slight distinction for
\(\widetilde G=\widetilde{\Spin}(2p+1,2p-1)\).
Since the cross action preserves the genuine central character,
different central characters always give distinct cross-action
orbits. The special feature in this case is that the underlying
Lie-level parameters in the two orbits are themselves different.
More precisely, the most split parameters
\[
\widetilde\gamma(\{i\}^*),
\qquad
i\in\{1,3,\ldots,2p-1\},
\]
have central character \(\chi_1\), whereas
\[
\widetilde\gamma(\{j\}^*),
\qquad
j\in\{2,4,\ldots,2p\},
\]
have central character \(\chi_2\).
Thus the two central-character blocks are represented by different
Lie-level cross-action orbits.
\end{remark}

The goal is to rule out the \(\rho/2\)-regular parameters that fail
to satisfy \eqref{e:key}. For each Cartan subgroup
$\widetilde H=\widetilde H^{r_+,r_-,m,s},$
we choose a representative
$\gamma=\gamma_{r_+,r_-,m,s}$ of the corresponding cross-action orbit as follows.

\begin{itemize}

\item[(a)] Suppose \(s\geq3\).
\begin{itemize}
    \item[(i)] If
$\widetilde G\in \{\widetilde{\Spin}(2p+2,2p),
\widetilde{\Spin}(2p+3,2p-1)\},$
or
$\widetilde G=\widetilde{\Spin}(2p+1,2p-1)$
with central character $\chi_1$,
we choose \(\gamma\) such that
\[
\theta_\gamma(x_j)=-x_j,
\qquad
2p-2\leq j\leq2p.
\]
\item[(ii)] If
$\widetilde G=\widetilde{\Spin}(2p+1,2p-1)$ with central character $\chi_2$,
we choose \(\gamma\) such that
\[
\theta_\gamma(x_j)=-x_j,
\qquad
2p-3\leq j\leq2p-1, \text{ and }
\]
\[
\theta_\gamma(x_{2p})=x_{2p}.
\]
\end{itemize}
\item[(b)] Suppose
$m\geq1,\ 1\leq s\leq2,\  r_+\geq1$.
\begin{itemize}
    \item[(i)]  If
$\widetilde G=\widetilde{\Spin}(2p+1,2p-1)$ with central character \(\chi_1\), we choose \(\gamma\) such that
\begin{align*}
\theta_\gamma(x_{2p-3})&=-x_{2p-2},
&
\theta_\gamma(x_{2p-2})&=-x_{2p-3},
\\
\theta_\gamma(x_{2p-1})&=x_{2p-1},
&
\theta_\gamma(x_{2p})&=-x_{2p}.
\end{align*}
\item[(ii)] If
$\widetilde G = \widetilde{\Spin}(2p+1,2p-1)
$ with central character \(\chi_2\), we choose \(\gamma\) such that
\begin{align*}
\theta_\gamma(x_{2p-3})&=-x_{2p-2},
&
\theta_\gamma(x_{2p-2})&=-x_{2p-3},
\\
\theta_\gamma(x_{2p-1})&=-x_{2p-1},
&
\theta_\gamma(x_{2p})&=x_{2p}.
\end{align*}
\item[(iii)] If
$\widetilde G=\widetilde{\Spin}(2p+2,2p)$ or $\widetilde{\Spin}(2p+3,2p-1)$, 
we choose \(\gamma\) such that
\begin{align*}
\theta_\gamma(x_{2p-2})&=-x_{2p-1},
&
\theta_\gamma(x_{2p-1})&=-x_{2p-2},
\\
\theta_\gamma(x_{2p})&=-x_{2p},
&
\theta_\gamma(x_{2p+1})&=x_{2p+1}.
\end{align*}
\end{itemize}
\item[(c)] Suppose
$\widetilde G=\widetilde{\Spin}(2p+2,2p),$ with
$s=0,\  m\geq1,\  r_-\geq1.$
In particular, \(r_+\geq2\). We choose \(\gamma\) such that
\begin{align*}
\theta_\gamma(x_{2p-2})&=-x_{2p-1},
&
\theta_\gamma(x_{2p-1})&=-x_{2p-2},
\\
\theta_\gamma(x_{2p})&=x_{2p},
&
\theta_\gamma(x_{2p+1})&=x_{2p+1}.
\end{align*}

\end{itemize}

\begin{lemma}\label{l:x-stab}
Let $\widetilde G\in
\{
\widetilde{\Spin}(2p+1,2p-1),
\widetilde{\Spin}(2p+2,2p),
\widetilde{\Spin}(2p+3,2p-1)
\},$ and let
\(\gamma=\gamma_{r_+,r_-,m,s}\)
be one of the representatives chosen above. Then the following
elements belong to the cross stabilizer \(W_\gamma\).

\begin{itemize}

\item[(a)] Suppose \(s\geq3\).
\begin{itemize}
    \item[(i)] If
    $\widetilde G\in\{
\widetilde{\Spin}(2p+2,2p),
\widetilde{\Spin}(2p+3,2p-1)\}$ or 
$\widetilde G=\widetilde{\Spin}(2p+1,2p-1)$ with central character $\chi_1$,
then $$s_{2p-2,2p}
\in
W_\gamma^r(\rho/2)
\subseteq W_\gamma.$$
\item[(ii)] If
$\widetilde G=\widetilde{\Spin}(2p+1,2p-1)$ with central character $\chi_2$,
then
$$s_{2p-3,2p-1}
\in
W_\gamma^r(\rho/2)
\subseteq W_\gamma.$$
\end{itemize}
\item[(b)] Suppose
$m\geq1,\  1\leq s\leq2,\ r_+\geq1.$
\begin{itemize}
    \item[(i)] If
$\widetilde G=\widetilde{\Spin}(2p+1,2p-1),$
with either central character \(\chi_1\) or \(\chi_2\), then
\[
w=
s_{2p-3,2p-1}s_{\overline{2p-3,2p-1}}
s_{2p-2,2p}s_{\overline{2p-2,2p}}
\in
W_\gamma^C(\rho/2)^\theta
\subseteq W_\gamma.
\]

\item[(ii)] If
$\widetilde G=
\widetilde{\Spin}(2p+2,2p)
$ or  $\widetilde{\Spin}(2p+3,2p-1),$
then
\[
w=
s_{2p-2,2p}s_{\overline{2p-2,2p}}
s_{2p-1,2p+1}s_{\overline{2p-1,2p+1}}
\in
W_\gamma^C(\rho/2)^\theta
\subseteq W_\gamma.
\]
\end{itemize}
\item[(c)] Suppose
$\widetilde G=\widetilde{\Spin}(2p+2,2p),
\  s=0,\  m\geq1,\ 
r_-\geq1.$
Then
\[
w=
s_{2p-2,2p}s_{\overline{2p-2,2p}}
s_{2p-1,2p+1}s_{\overline{2p-1,2p+1}}
\in
W_\gamma^C(\rho/2)^\theta
\subseteq W_\gamma.
\]

\end{itemize}

Here \(s_{i,j}\) denotes the reflection
\(s_{e_i-e_j}\), whereas
\(s_{\overline{i,j}}\) denotes the reflection
\(s_{e_i+e_j}\).

\begin{proof}
Case (a) is immediate, since
\(e_{2p-2}-e_{2p}\), or respectively
\(e_{2p-3}-e_{2p-1}\), is real for the chosen parameter.

For cases (b) and (c), by the description of the parameters in
Section~\ref{s:parameters}, it suffices to verify that the indicated
element \(w\) preserves the Cartan involution, namely
$\theta_\gamma w=w\theta_\gamma.$

We verify this in case (b) for
\(\widetilde G=\widetilde{\Spin}(2p+1,2p-1)\)
with central character \(\chi_1\); the other cases are similar.
We have
\begin{align*}
&\theta_\gamma w
(\ldots,x_{2p-3},x_{2p-2},x_{2p-1},x_{2p})
\\
&\qquad
=
\theta_\gamma
(\ldots,-x_{2p-3},-x_{2p-2},-x_{2p-1},-x_{2p})
\\
&\qquad
=
(\ldots,x_{2p-2},x_{2p-3},-x_{2p-1},x_{2p}),
\end{align*}
whereas
\begin{align*}
&w\theta_\gamma
(\ldots,x_{2p-3},x_{2p-2},x_{2p-1},x_{2p})
\\
&\qquad
=
w(\ldots,-x_{2p-2},-x_{2p-3},x_{2p-1},-x_{2p})
\\
&\qquad
=
(\ldots,x_{2p-2},x_{2p-3},-x_{2p-1},x_{2p}).
\end{align*}
Thus
$\theta_\gamma w=w\theta_\gamma,$
and hence
$w\in W_\gamma^C(\rho/2)^\theta.$
\end{proof}
\end{lemma}

\begin{lemma}\label{l:D-ruleout}
Let $\widetilde G\in
\{
\widetilde{\Spin}(2p+1,2p-1),
\widetilde{\Spin}(2p+2,2p),
\widetilde{\Spin}(2p+3,2p-1)
\},$
and let
$\widetilde H=\widetilde H^{r_+,r_-,m,s}$
be a Cartan subgroup. Choose
\(\gamma=\gamma_{r_+,r_-,m,s}\)
as above. Then \(\gamma\) fails to satisfy \eqref{e:key} in each of
the following cases:
\begin{itemize}
\item[(a)] \(s\geq3\);
\item[(b)] \(m\geq1\), \(1\leq s\leq2\), and \(r_+\geq1\);
\item[(c)] \(\widetilde G=\widetilde{\Spin}(2p+2,2p)\),
\(s=0\), \(m\geq1\), and \(r_-\geq1\).
\end{itemize}

\begin{proof}
We exhibit in each case an element \(w\in W_\gamma\) for which
\[
\epsilon_\gamma(w)
\neq
\operatorname{sgn}_{W(\rho/2)}(w).
\]

\smallskip

\noindent
\textit{Case (a).}
Suppose \(s\geq3\).

First consider $\widetilde G=\widetilde{\Spin}(2p+2,2p),
\ \widetilde{\Spin}(2p+3,2p-1)$, or 
$\widetilde G=\widetilde{\Spin}(2p+1,2p-1)$
with central character \(\chi_1\).

By Lemma~\ref{l:x-stab},
$w=s_{2p-2,2p}\in W_\gamma.$
We may write
$w=s_{\beta_3}s_{\beta_2}s_{\beta_1},$
where
\[
\beta_1=e_{2p-2}-e_{2p-1},\qquad
\beta_2=e_{2p-2}-e_{2p},\qquad
\beta_3=e_{2p-1}-e_{2p}.
\]
All three roots are real for their respective preceding parameters.
Hence
\[
m(\gamma,w)=0,
\qquad
\epsilon_\gamma(w)=1.
\]
On the other hand,
$\operatorname{sgn}_{W(\rho/2)}(w)=-1.$
Thus \eqref{e:key} fails.

Now consider
$\widetilde G=\widetilde{\Spin}(2p+1,2p-1)$
with central character \(\chi_2\). By Lemma~\ref{l:x-stab},
$w=s_{2p-3,2p-1}\in W_\gamma.$
Write
$w=s_{\beta_3}s_{\beta_2}s_{\beta_1},$
where
\[
\beta_1=e_{2p-3}-e_{2p-2},\qquad
\beta_2=e_{2p-3}-e_{2p-1},\qquad
\beta_3=e_{2p-2}-e_{2p-1}.
\]
Again all three roots are real for their respective preceding
parameters, so
\[
\epsilon_\gamma(w)=1,
\qquad
\operatorname{sgn}_{W(\rho/2)}(w)=-1.
\]
Hence \eqref{e:key} fails.

\smallskip

\noindent
\textit{Case (b).}
Suppose
$m\geq1,\  1\leq s\leq2,\  r_+\geq1.$

First let
$\widetilde G=\widetilde{\Spin}(2p+1,2p-1),$ with central character \(\chi_1\) or \(\chi_2\). By
Lemma~\ref{l:x-stab},
\[
w=
s_{2p-3,2p-1}s_{\overline{2p-3,2p-1}}
s_{2p-2,2p}s_{\overline{2p-2,2p}}
\in W_\gamma.
\]
We decompose
\[
w=s_{\beta_{20}}s_{\beta_{19}}\cdots s_{\beta_1},
\]
where the roots \(\beta_j\), in order, are
\begin{align*}
&e_{2p-1}+e_{2p},
e_{2p-1}-e_{2p},
e_{2p-3}+e_{2p-2},
e_{2p-3}+e_{2p-1},
e_{2p-2}+e_{2p-1},
\\
&e_{2p-3}+e_{2p},
e_{2p-3}-e_{2p},
e_{2p-2}+e_{2p-1},
e_{2p-3}-e_{2p-1},
e_{2p-3}+e_{2p-2},
\\
&e_{2p-2}+e_{2p-1},
e_{2p-2}-e_{2p},
e_{2p-1}+e_{2p},
e_{2p-2}-e_{2p-1},
e_{2p-2}+e_{2p-1},
\\
&e_{2p-1}-e_{2p},
e_{2p-2}+e_{2p},
e_{2p-2}+e_{2p-1},
e_{2p-1}-e_{2p},
e_{2p-1}+e_{2p}.
\end{align*}
For either central character \(\chi_1\) or \(\chi_2\),
\(\beta_3\) is the only imaginary root for its respective preceding
parameter. Therefore
\[
m(\gamma,w)=1,
\qquad
\epsilon_\gamma(w)=-1.
\]
Since \(w\) is expressed as a product of \(20\) simple reflections,
$\operatorname{sgn}_{W(\rho/2)}(w)=1.$ Thus \eqref{e:key} fails.

For
$\widetilde G=
\widetilde{\Spin}(2p+2,2p)$ or $\widetilde{\Spin}(2p+3,2p-1),$
the argument is analogous. Using the element
\[
w=
s_{2p-2,2p}s_{\overline{2p-2,2p}}
s_{2p-1,2p+1}s_{\overline{2p-1,2p+1}}
\]
from Lemma~\ref{l:x-stab}, the corresponding decomposition has
\(m(\gamma,w)\) odd, whereas
$$\operatorname{sgn}_{W(\rho/2)}(w)=1.$$
Hence \eqref{e:key} again fails.

\smallskip

\noindent
\textit{Case (c).}
Suppose
$\widetilde G=\widetilde{\Spin}(2p+2,2p),
\ s=0,\qquad
m\geq1,\qquad
r_-\geq1.$ 
By Lemma~\ref{l:x-stab},
\[
w=
s_{2p-2,2p}s_{\overline{2p-2,2p}}
s_{2p-1,2p+1}s_{\overline{2p-1,2p+1}}
\in W_\gamma.
\]
Write
\[
w=s_{\beta_{20}}s_{\beta_{19}}\cdots s_{\beta_1},
\]
where the roots \(\beta_j\), in order, are
\begin{align*}
&e_{2p}+e_{2p+1},
e_{2p}-e_{2p+1},
e_{2p-2}+e_{2p-1},
e_{2p-2}+e_{2p},
e_{2p-1}+e_{2p},
\\
&e_{2p-2}+e_{2p+1},
e_{2p-2}-e_{2p+1},
e_{2p-1}+e_{2p},
e_{2p-2}-e_{2p},
e_{2p-2}+e_{2p-1},
\\
&e_{2p-1}+e_{2p},
e_{2p-1}-e_{2p+1},
e_{2p}+e_{2p+1},
e_{2p-1}-e_{2p},
e_{2p-1}+e_{2p},
\\
&e_{2p}-e_{2p+1},
e_{2p-1}+e_{2p+1},
e_{2p-1}+e_{2p},
e_{2p}-e_{2p+1},
e_{2p}+e_{2p+1}.
\end{align*}
The roots
\[
\beta_1,\ \beta_2,\ \beta_3,\ \beta_{13},\
\beta_{16},\ \beta_{19},\ \beta_{20}
\]
are imaginary for their respective preceding parameters. Thus
\[
m(\gamma,w)=7,
\qquad
\epsilon_\gamma(w)=-1.
\]
On the other hand,
\[
\operatorname{sgn}_{W(\rho/2)}(w)=1.
\]
Hence \eqref{e:key} fails.
\end{proof}
\end{lemma}

Theorem~\ref{t:small-exhaustion} follows from the following
proposition.

\begin{prop}\label{p:RD2}
Fix a \(\rho/2\)-compatible genuine central character \(\chi\) of
\(\widetilde G\). Then
\[
\left|
\prod\nolimits_{\rho/2}^{s}(\widetilde G)_\chi
\right|
=
\begin{cases}
1,
&
\widetilde G=\widetilde{\Spin}(2p+1,2p-1),
\\[1mm]
3,
&
\widetilde G=\widetilde{\Spin}(2p+2,2p),
\\[1mm]
1,
&
\widetilde G=\widetilde{\Spin}(2p+3,2p-1).
\end{cases}
\]

\begin{proof}
By Lemma~\ref{l:D-ruleout}, the Cartan subgroups not ruled out in
the three cases are respectively
\[
\begin{array}{c|c}
\widetilde G
&
\text{Cartan subgroups not ruled out}
\\
\hline
\widetilde{\Spin}(2p+1,2p-1)
&
\widetilde H^{p,p-1,0,1}
\\[1mm]
\widetilde{\Spin}(2p+2,2p)
&
\widetilde H^{p,p-1,0,2},\
\widetilde H^{p+1,p,0,0},\
\widetilde H^{1,0,p,0}
\\[1mm]
\widetilde{\Spin}(2p+3,2p-1)
&
\widetilde H^{p+1,p-1,0,1}.
\end{array}
\]
By Remark~\ref{r:orbit}, after fixing \(\chi\), each of these Cartan
subgroups contributes at most one cross-action orbit satisfying
\eqref{e:key}. Hence
\[
\left|
\prod\nolimits_{\rho/2}^{s}(\widetilde G)_\chi
\right|
\leq
1,\ 3,\ 1
\]
in the three cases, respectively.

On the other hand, Proposition~\ref{p:RD1} gives
\[
\mathcal R_{\rho/2}(\widetilde G)_\chi
\subseteq
\prod\nolimits_{\rho/2}^{s}(\widetilde G)_\chi,
\]
and Definition~\ref{d:list} gives
\[
\left|
\mathcal R_{\rho/2}(\widetilde G)_\chi
\right|
=
1,\ 3,\ 1,
\]
respectively. The asserted equalities follow.
\end{proof}
\end{prop}

\section{Lift of the Trivial Representation}

\subsection{Preliminaries on lifting}

We recall the facts about Kazhdan--Patterson lifting that will be used below. We refer the reader to \cite{AHe10} and \cite[Sections 3 and 4]{Ts23} for the construction of the transfer factor and for the detailed definitions.

Let $G$ be a real form of a connected, simply connected, semisimple complex group of simply laced type, and let $\tu G$ be its nonlinear double cover. Let $p:\tu G\to G$ be the covering map. If $\pi$ is a stable admissible representation of $G$, the Kazhdan--Patterson lifting operator defines a genuine virtual representation of $\tu G$, denoted
$\Lift_G^{\tu G}(\pi).$

At the level of characters, this is given by
\begin{equation*}
\Lift_G^{\tu G}(\Theta_\pi)(\tu g)
=
\sum_{ \{h\in G\mid h^2=p(\tu g)\}}
\Delta(h,\tu g)\Theta_\pi(h),
\end{equation*}
where $\Delta(h,\tu g)$ is the canonical transfer factor.

If
\begin{equation}\label{e:lift}
\Lift_G^{\tu G}(\pi)
=\sum_{\tu\pi} a_{\tu\pi}\tu\pi,
\end{equation}
where $\tu\pi$ runs over irreducible genuine admissible representations of $\tu G$, we define
\begin{equation*}
\LLift (\pi) =
\{\tu\pi\mid a_{\tu\pi}\neq 0 \text{ in \eqref{e:lift}}\}
\end{equation*}
as a set.  Thus $\LLift(\pi)$ is the set of irreducible genuine constituents occurring in the lift.
\medskip 

We will use the lifting operator on stable sums of standard modules. Let
\begin{equation*}
\gamma=(H,\Gamma,\overline\gamma)\in CD(G,H)
\end{equation*}
be a modified regular character of $G$. Write $H=TA$, let $M=\operatorname{Cent}_G(A)$, and let $W_i$ be the Weyl group of the imaginary root system for $H$. We define the stable standard module attached to $\gamma$ by
\begin{equation*}
I_G^{\mathrm{st}}(\gamma)
=
\sum_{w\in W(M,H)\backslash W_i} I_G(w\gamma).
\end{equation*}

\begin{thm}\label{t:lift-stable-standard}
\emph{(\cite[Corollary 19.8]{AHe10})}
Let $\gamma=(H,\Gamma,\overline\gamma)\in CD(G,H)$ be a modified regular character. Let
\begin{equation*}
\{\tu\gamma_1,\ldots,\tu\gamma_n\}
\end{equation*}
be the set of constituents of $\Lift_G^{\tu G}(w\gamma)$ as $w$ runs over $W_i$, counted without multiplicity. Then
\begin{equation*}
\Lift_G^{\tu G}(I_G^{\mathrm{st}}(\gamma))
=
C(H)\sum_{i=1}^{n} I_{\tu G}(\tu\gamma_i),
\end{equation*}
where
\begin{equation*}
C(H)=\frac{c(H)}{c(H_s)},
\qquad
c(H)=|H_2^0|\cdot |H/Z_0(H)|^{1/2}.
\end{equation*}
Here $H_s$ is the maximally split Cartan subgroup of $G$, $H_2^0$ is the subgroup of elements of order $2$ in the identity component of $H$, and
\begin{equation*}
Z_0(H)=p(Z(\tu H)).
\end{equation*}
In particular, $C(H_s)=1$.
\end{thm}

We will also use the following consequence of lifting and coherent continuation.

\begin{thm} \cite[Theorem 4.1]{Ts23}  \label{t:lift-trivial-small} 
Let $\bbC$ be the trivial representation of $G$. If
$\tu\pi\in \Lift (\bbC),$ then $\tu\pi$ has infinitesimal character $\rho/2$ and maximal $\tau$-invariant. Equivalently,
\begin{equation*}
\LLift (\bbC)
\subseteq
\prod\nolimits_{\rho/2}^{s}(\tu G).
\end{equation*}
\end{thm}

\subsection{Lift of the trivial representation for type $D$}
Throughout this subsection, $$\tu G\in\{\tu\Spin(2p+1,2p-1),\tu\Spin(2p+2,2p), \tu\Spin(2p+3, 2p-1)\}.$$ We will show which  genuine small representations with infinitesimal character $\rho/2$ occur in $\Lift_G^{\tu G}(\bbC)$.

\begin{definition} \label{d:param}
Fix a \(\rho/2\)-compatible genuine central character \(\chi\).
For the set
\(\mathcal R_{\rho/2}(\widetilde G)_\chi\)
defined in Definition~\ref{d:list}, let
\[
\mathcal P_{\rho/2}(\widetilde G)_\chi
=
\left\{
\widetilde\gamma \mid
J(\widetilde\gamma)
\in
\mathcal R_{\rho/2}(\widetilde G)_\chi
\right\}
\]
be the corresponding set of \(\rho/2\)-regular characters.

Define the following sets of strongly orthogonal roots:
\begin{align*}
S_1 &=\{e_1-e_2, e_3-e_4, \dots, e_{2p-3}-e_{2p-2}\},\\
S_2 &=\{e_1-e_2, e_3-e_4, \dots, e_{2p-3}-e_{2p-2}, e_{2p-1}-e_{2p}\},\\
S_3 &=\{e_1-e_2, e_3-e_4, \dots, e_{2p-3}-e_{2p-2}, e_{2p-1}+e_{2p}\}.
\end{align*}
 Then $\mathcal P_{\rho/2}(\tu G)_\chi$ is described as follows.

\begin{itemize}
    \item[(a)] If $\tu G=\tu\Spin (2p+1,2p-1)$, let $\tu\gamma_\sharp =\tu \gamma(\{2p-1\}^*)$. Then  
    $$\mathcal P_{\rho/2}(\tu G)_\chi=
    \begin{cases}
        \{c_{S_1}(\tu\gamma_\sharp)\} & \text{ if } \chi = \chi_1, \\ \{\tu\gamma(\{2p\}^*)\} & \text{ if } \chi=\chi_2.
      \end{cases}
    $$
  \item[(b)] If $\tu G=\tu\Spin (2p+2,2p)$,
  let  $\tu\gamma_\sharp= \tu\gamma(\{2p+1\}^*)_\chi$. Then 
    $$\mathcal P_{\rho/2}(\tu G)_\chi=\{\tu \gamma_\sharp, c_{S_2}(\tu \gamma_\sharp), c_{S_3}(\tu \gamma_\sharp)\}.$$
    \item[(c)] If $\tu G=\tu\Spin (2p+3,2p-1)$, let $\tu\gamma_\sharp =\tu \gamma(\{2p-1, 2p+1\}^*)$. Then 
    $$\mathcal P_{\rho/2}(\tu G)_\chi=\{c_{S_1}(\tu\gamma_\sharp)\}.$$
\end{itemize}

\end{definition}

\begin{notation} \label{n:linear-param}
We need a notation for the $\rho$-regular characters of the linear group $G$.
In Section \ref{s:parameters}, we used $\tu\gamma(\mathcal S)$ to denote a
$\rho/2$-regular character of $\tu G$. For the linear group, we will use the
notation
\begin{equation}\label{e:linear-param}
\gamma(\mathcal S;\varepsilon)
\end{equation}
to denote a $\rho$-regular character
$\gamma=(H,\Gamma,\overline\gamma)$
of $G$. Here the symbol $\mathcal S$ has the same meaning as in Section
\ref{s:parameters}, while
\[
\varepsilon=(\varepsilon_1,\ldots,\varepsilon_s)
=
\Gamma|_{H/H^0}
\]
records the character of the component group of $H$. After choosing the
standard generators of $H/H^0$, we write each
\[
\varepsilon_i\in\{\pm1\}.
\]
These signs are attached to the real split coordinates which are not contained
in the set $\mathcal S$. This extra datum is necessary for the linear group
because it determines the parity or nonparity of real roots.
\end{notation}

The next lemma gives the $\rho$-regular character whose Langlands quotient is
the trivial representation for each linear group under consideration.

\begin{lemma}\label{l:triv-param}
Let
\[
G\in
\{\Spin(2p+1,2p-1),\Spin(2p+2,2p),\Spin(2p+3,2p-1)\}.
\]
Let $\gamma_0$ be the $\rho$-regular character of $G$ such that
$J(\gamma_0)\simeq \bbC.$

Then, using the notation of \eqref{e:linear-param}, we have the following.

\begin{itemize}
    \item[(a)] If $G=\Spin(2p+1,2p-1)$, then
    \[
    \gamma_0=\gamma(\{2p\}^*;\varepsilon_{\mathrm{triv}}),
    \qquad
    \varepsilon_{\mathrm{triv}}
    =
    \underbrace{(+1,\ldots,+1)}_{2p-1}.
    \]

    \item[(b)] If $G=\Spin(2p+2,2p)$, then
    \[
    \gamma_0=\gamma(\{2p+1\}^*;\varepsilon_{\mathrm{triv}}),
    \qquad
    \varepsilon_{\mathrm{triv}}
    =
    \underbrace{(+1,\ldots,+1)}_{2p}.
    \]

    \item[(c)] If $G=\Spin(2p+3,2p-1)$, then
    \[
    \gamma_0=\gamma(\{2p-1,2p+1\}^*;\varepsilon_{\mathrm{triv}}),
    \qquad
    \varepsilon_{\mathrm{triv}}
    =
    \underbrace{(+1,\ldots,+1)}_{2p-1}.
    \]
\end{itemize}

\begin{proof}
Let $H_s=T_sA_s$ be the maximally split Cartan subgroup of $G$, and let
$P_s=M_sA_sN_s$ be the corresponding minimal parabolic subgroup. The trivial
representation of $G$ is the Langlands quotient of the spherical standard
module attached to $P_s$.

In the parametrization of regular characters, the parameter
$\overline\gamma_0$ decomposes into an $M_s$-part and an $A_s$-part. The
$M_s$-part is the Harish-Chandra parameter of the representation of $M_s$
appearing in the inducing data. For the trivial representation of $G$, this is
the smallest possible Harish-Chandra parameter on the $M_s$-side. With our
choice of coordinates, these smallest Harish-Chandra parameters are
\[
0,\qquad 0,\qquad (1,0)
\]
for
\[
\Spin(2p+1,2p-1),\qquad
\Spin(2p+2,2p),\qquad
\Spin(2p+3,2p-1),
\]
respectively.

Thus, in the notation fixed in Section \ref{s:parameters}, the corresponding
symbols for the maximally split parameters are
\[
\{2p\}^*,\quad
\{2p+1\}^*,\quad
\{2p-1,2p+1\}^*,
\]
respectively.

Moreover, since the standard module is spherical, the character on the
component group of $H_s$ is trivial. Hence
\[
\Gamma_0|_{H_s/H_s^0}=1,
\]
or equivalently
\[
\varepsilon_{\mathrm{triv}}=(+1,\ldots,+1).
\]
The number of entries is the number of real split coordinates not contained in
the symbol $\mathcal S$.  Therefore, we obtain the desired parameters $\gamma_0$ as in the statement. 

\end{proof}

\end{lemma}

\begin{lemma} \label{l:triv-M}
Let $G\in \{\Spin(2p+1,2p-1), \Spin(2p+2,2p), \Spin(2p+3, 2p-1)\}$. Let $\gamma_0$ be the $\rho$-regular character of the trivial representation (see Lemma \ref{l:triv-param}). 
Consider $S\in \{ S_1, S_2, S_3\}$ as defined in Definition \ref{d:param} (depending on $G$). Let $S'\subseteq S$. Then 
$$
M(c_{S'}(\gamma_0) , \gamma_0)= (-1)^{\ell(\gamma_0) - \ell(c_{S'}(\gamma_0))}.
$$
\begin{proof}
    This can be checked by Kazhdan--Lusztig--Vogan algorithm for linear groups 
   (see \cite[Proposition 6.14]{V83a}).  
\end{proof}

\end{lemma}

In the next Lemma, we characterize the standard modules that contain $J(\tu\gamma)$ as a composition factor for some $\gamma \in \mathcal P_{\rho/2}(\tu G)$. 
The proof is parallel to that of \cite[Lemma 7.2]{Ts23}.

\begin{lemma}\label{l:comp-factor}
Let  $S_1, S_2, S_3$ and $\tu\gamma_\sharp$ as  in Definition  \ref{d:param}. 
\begin{itemize}
    \item[(1)] If $\tu G\in\{\tu\Spin(2p+1, 2p-1), \tu\Spin(2p+3, 2p-1)\}$, then $J(c_{S_1}(\tu \gamma_\sharp))$ is a composition factor of $I(\tu \gamma)$ if and only if 
    $$
\tu\gamma = c_{S'}(\tu\gamma_\sharp) 
    $$
  for some subset  $S'\subseteq S_1$. 
  \item[(2)] If $\tu G=\tu\Spin(2p+2, 2p)$, then for $j\in\{2, 3\}$, $J(c_{S_j}(\tu \gamma_\sharp))$ is a composition factor of $I(\tu \gamma)$ if and only if 
    $$
\tu\gamma = c_{S'}(\tu\gamma_\sharp) 
    $$
   for some subset $S'\subseteq S_j$. 
\end{itemize}
Moreover, in each case, the multiplicity of $J(c_{S_j}(\tu\gamma_\sharp) )$ in $I(c_{S'}(\tu\gamma_\sharp))$ is 1, that is,
$$
m(c_{S_j} (\tu\gamma_\sharp) , c_{S'}(\tu\gamma_\sharp)) =1. 
$$

\end{lemma}
Fix a $\rho/2$-compatible genuine central character
$\chi\in \widehat Z_{\rho/2}(\tu G).$ Let $\LLift(\bbC)_\chi$ denote the set of irreducible genuine representations
with central character $\chi$ occurring in the lift of the trivial representation.
We have the following description of $\LLift(\bbC)_\chi$.

\begin{thm}\label{t:main}
Using the notation in Definition \ref{d:list}, and fixing
$\chi\in \widehat Z_{\rho/2}(\tu G)$, the set $\LLift(\bbC)_\chi$ is described
as follows.
\begin{itemize}
    \item[(a)] If $\tu G=\tu{\Spin}(2p+1,2p-1)$, then
    \[
    \LLift(\bbC)_\chi
    =\begin{cases}
        \emptyset, & \ \chi =\chi_1\\
    \{J(\{2p\}^*)\}, & \chi=\chi_2.
    \end{cases}
    \]

    \item[(b)] If $\tu G=\tu{\Spin}(2p+2,2p)$, then
    \[
    \LLift(\bbC)_\chi
    =
    \calR_{\rho/2}(\tu G)_\chi.
    \]

    \item[(c)] If $\tu G=\tu{\Spin}(2p+3,2p-1)$, then
    \[
    \LLift(\bbC)_\chi
    =
    \emptyset.
    \]
\end{itemize}
Here the representations appearing in the right-hand side are understood to
have central character $\chi$.

In particular, 
\[
\LLift(\bbC)_\chi
=
\prod\nolimits_{\rho/2}^{s}(\tu G)_\chi
\]
for  $\tu G=\tu{\Spin}(2p+2,2p)$ and for $\tu G=\tu\Spin (2p+1,2p-1)$ with $\chi=\chi_2$. 

For $\tu G=\tu\Spin (2p+1,2p-1)$ with $\chi=\chi_1$ and for
$\tu G=\tu\Spin(2p+3,2p-1)$,
$\LLift(\bbC)_\chi$ is a proper subset of
$\prod\nolimits_{\rho/2}^{s}(\tu G)_\chi.$

Consequently,
\[
\LLift(\bbC)
=
\prod\nolimits_{\rho/2}^{s}(\widetilde G)
\]
if and only if
\(\widetilde G=\widetilde{\Spin}(2p+2,2p)\).
For \(\widetilde{\Spin}(2p+1,2p-1)\), the lift exhausts the
\(\chi_2\)-block but vanishes on the \(\chi_1\)-block, whereas for
\(\widetilde{\Spin}(2p+3,2p-1)\) the lift is empty.

\end{thm}

Since
$\LLift(\bbC)\subseteq \mathcal R_{\rho/2}(\tu G),$
to prove Theorem \ref{t:main}, it suffices to compute the coefficient of
$J(\tu\delta)$ in $\Lift_G^{\tu G}(\bbC)$ for
$\tu\delta\in \mathcal P_{\rho/2}(\tu G)$.
We first prove the following lemma.

\begin{lemma} \label{l:coefficient-formula}
Let $\tu\delta\in\mathcal P_{\rho/2}(\tu G)$.
Let $\gamma_0$ be the $\rho$-regular character of $G$ such that
$J(\gamma_0)=\bbC$ (see Lemma \ref{l:triv-param}). Suppose that the character formula for $J(\gamma_0)$ has
been grouped into stable standard sums:
\begin{equation} \label{e:char-formular-C}
J(\gamma_0)
=
\sum_{\gamma}
M(\gamma,\gamma_0)I_G^{\mathrm{st}}(\gamma), 
\end{equation}
where $M(\gamma,\gamma_0)= (-1)^{\ell(\gamma_0) -\ell(\gamma)}$. 
For each such $\gamma$, write
\begin{equation}
\Lift_G^{\tu G}(I_G^{\mathrm{st}}(\gamma))
=
C(H_\gamma)
\sum_{\tu\gamma\in\mathcal L(\gamma)}
I_{\tu G}(\tu\gamma),
\end{equation}
where $\mathcal L(\gamma)$ denotes the full set of genuine standard parameters
occurring in the lift and $C(H_\gamma)$ is the lifting constant from Theorem \ref{t:lift-stable-standard}. Then the coefficient of $J(\tu\delta)$ in
$\Lift_G^{\tu G}(\bbC)$ is
\begin{equation} \label{e:coeff}
\sum_{\gamma}
M(\gamma,\gamma_0)C(H_\gamma)
\sum_{\tu\gamma\in\mathcal L(\gamma)}
m(\tu\delta,\tu\gamma).
\end{equation}
\end{lemma}

\begin{proof}
Applying the lifting operator to \eqref{e:char-formular-C}, we get
\begin{align*}
\Lift_G^{\tu G}(\bbC)
&=
\sum_{\gamma}
M(\gamma,\gamma_0)\,
\Lift_G^{\tu G}(I_G^{\mathrm{st}}(\gamma))\\
&= \sum_{\gamma}
M(\gamma,\gamma_0)\, C(H_\gamma)
\sum_{\tu\gamma\in \mathcal L(\gamma)}
I_{\tu G}(\tu\gamma),
\end{align*}
where the second equality follows from Theorem \ref{t:lift-stable-standard}.
Since $I_{\tu G}(\tu\gamma)=
\sum\limits_{\tu\eta}
m(\tu\eta,\tu\gamma)[J(\tu\eta)],$ collecting the coefficient of $J(\tu\delta)$ gives 
\begin{equation*}
\sum_{\gamma}
M(\gamma,\gamma_0)\,C(H_\gamma)
\sum_{\tu\gamma\in \mathcal L(\gamma)}
m(\tu\delta,\tu\gamma).
\end{equation*}
This proves the lemma.
\end{proof}

\begin{proof}[Proof of Theorem~\ref{t:main}]
We use the notation in Definition \ref{d:param}.

We first prove (a). Consider $\tu G =\tu\Spin(2p+1,2p-1)$. 
Let
$\tu\delta=\tu\gamma(\{2p\}^*)$.  This parameter has maximal length among the $\rho/2$-regular characters in the
genuine block.  It can be computed that for $\tu\gamma=\tu\gamma(\{i\}^*)$ with $\ell(\tu\gamma)=\ell (\tu\delta)$, $m(\tu\delta,\tu\gamma)=0$ unless $\tu\gamma=\tu\delta$.

Thus, in \eqref{e:coeff}, the only lifted standard module that can contribute
to $J(\tu\delta)$ is $I_{\tu G}(\tu\delta)$. This standard module is obtained
from the stable parameter $\gamma=\gamma_0$. Hence
$M(\gamma_0,\gamma_0)=1.$
Moreover, 
$C(H_{\gamma_0})=1.$
Therefore the coefficient of
$J(\tu\gamma(\{2p\}^*))$ in $\Lift_G^{\tu G}(\bbC)$ is $$M(\gamma_0,\gamma_0)C(H_{\gamma_0})m(\tu\delta,\tu\delta)=
1.$$
Hence
$$
J(\tu\gamma(\{2p\}^*))
\in \LLift(\bbC).
$$

Now let
\[
\tu\delta=c_{S_1}(\tu\gamma_\sharp).
\]
For each subset \(S'\subseteq S_1\), define
\[
\gamma_{S'}:=c_{S'}(\gamma_0).
\]
By Lemmas \ref{l:triv-M}, \ref{l:comp-factor} and \ref{l:coefficient-formula}, the coefficient of
\(J(\tu\delta)\) in \(\Lift_G^{\tu G}(\bbC)\) is
\begin{equation}\label{e:coeff1}
K_{S_1}
=
\sum_{S'\subseteq S_1}
(-1)^{\ell(\gamma_0)-\ell(\gamma_{S'})}
C(H_{S'}),
\end{equation}
where \(H_{S'}\) is the Cartan subgroup attached to \(\gamma_{S'}\), and
\(C(H_{S'})\) is the corresponding lifting constant.

Write
\[
S_1=\{\alpha_1,\alpha_3,\ldots,\alpha_{2p-3}\},
\qquad
\alpha_i=e_i-e_{i+1}.
\]
Thus \(|S_1|=p-1\). Table \ref{tab:C-H-spin-2p1-2p1} gives the relevant Cartan subgroups and
lifting constants. We use \atlas\ to identify the Cartan classes \(H_{S'}\).

\begin{table}[!htbp]
\centering
\renewcommand{\arraystretch}{1.35}
\[
\begin{array}{c|c|c|c}
S' & \#\text{ of such }S' & H_{S'} & C(H_{S'}) \\
\hline
\emptyset
&
1
&
(\bbR^\times)^{2p-2}\times \bbC^\times
&
1
\\

\{\alpha_i\}
&
{p-1\choose 1}
&
(\bbR^\times)^{2p-4}\times(\bbC^\times)^2
&
1
\\

\vdots
&
\vdots
&
\vdots
&
\vdots
\\

\{\alpha_{i_1},\ldots,\alpha_{i_k}\}
&
{p-1\choose k}
&
(\bbR^\times)^{2p-2-2k}\times(\bbC^\times)^{k+1}
&
1
\\

\vdots
&
\vdots
&
\vdots
&
\vdots
\\

S_1-\{\alpha_i\}
&
{p-1\choose p-2}
&
(\bbR^\times)^2\times(\bbC^\times)^{p-1}
&
1
\\

S_1
&
1
&
(\bbC^\times)^p
&
1
\end{array}
\]
\caption{The Cartan subgroups \(H_{S'}\) and lifting constants \(C(H_{S'})\) for \(S'\subseteq S_1\) in the case \(G=\Spin(2p+1,2p-1)\).}
\label{tab:C-H-spin-2p1-2p1}
\end{table}

In Table \ref{tab:C-H-spin-2p1-2p1},  \(i\in\{1,3,\ldots,2p-3\}\), and $\{i_1,\ldots,i_k\}\subseteq \{1,3,\ldots,2p-3\}$.
For \(|S'|=k\), we have $
\ell(\gamma_0)-\ell(\gamma_{S'})=k,$
and the table gives
$C(H_{S'})=1.$
Therefore \eqref{e:coeff1} becomes
\[
K_{S_1}
=
\sum_{k=0}^{p-1}
(-1)^k {p-1\choose k}
=
(1-1)^{p-1}
=
0.
\]
Hence
\[
J(\{1,2\},\ldots,\{2p-3,2p-2\},\{2p-1\}^*)
\notin \LLift(\bbC).
\]

Now we look at (b).  Consider $\tu G=\tu\Spin (2p+2,2p)$. 
Using the same argument of the first part of the proof of (a), we can show  that 
the coefficient of $J(\{2p+1\}^*)$ in $\LiftC$ is 1.

Now let $\tu\delta=c_{S_2}(\tu\gamma_\sharp).$
For each subset \(S'\subseteq S_2\), define
\[
\gamma_{S'}:=c_{S'}(\gamma_0).
\]
We compute the coefficient \(K_{S_2}\) as in \eqref{e:coeff1}, with
\(S_1\) replaced by \(S_2\).

Write
\[
S_2=\{\alpha_1,\alpha_3,\ldots,\alpha_{2p-1}\},
\qquad
\alpha_i=e_i-e_{i+1}.
\]
Thus \(|S_2|=p\). The relevant Cartan subgroups and lifting constants are
listed in Table \ref{tab:C-H-spin-2p2-2p}.

\begin{table}[!htbp]
\centering
\renewcommand{\arraystretch}{1.35}
\[
\begin{array}{c|c|c|c}
S' & \#\text{ of such }S' & H_{S'} & C(H_{S'}) \\
\hline
\emptyset
&
1
&
(\bbR^\times)^{2p-1}\times \bbC^\times
&
1
\\

\{\alpha_i\}
&
{p\choose 1}
&
(\bbR^\times)^{2p-3}\times(\bbC^\times)^2
&
1
\\

\vdots
&
\vdots
&
\vdots
&
\vdots
\\

\{\alpha_{i_1},\ldots,\alpha_{i_k}\}
&
{p\choose k}
&
(\bbR^\times)^{2p-1-2k}\times(\bbC^\times)^{k+1}
&
1
\\

\vdots
&
\vdots
&
\vdots
&
\vdots
\\

S_2-\{\alpha_i\}
&
{p\choose p-1}
&
\bbR^\times\times(\bbC^\times)^p
&
1
\\

S_2
&
1
&
S^1\times(\bbC^\times)^p
&
2
\end{array}
\]
\caption{The Cartan subgroups \(H_{S'}\) and lifting constants \(C(H_{S'})\) for \(S'\subseteq S_2\) in the case \(G=\Spin(2p+2,2p)\).}
\label{tab:C-H-spin-2p2-2p}
\end{table}

In Table \ref{tab:C-H-spin-2p2-2p}, we have
$i\in\{1,3,\ldots,2p-1\}$ and $\{i_1,\ldots,i_k\}\subseteq\{1,3,\ldots,2p-1\}.$
For \(|S'|=k\), we have
$\ell(\gamma_0)-\ell(\gamma_{S'})=k.$
Moreover, Table \ref{tab:C-H-spin-2p2-2p} gives
\[
C(H_{S'})=1
\qquad\text{for } |S'|<p,
\]
while
\[
C(H_{S_2})=2.
\]
Therefore
$$
K_{S_2}
=
\sum_{k=0}^{p-1}
(-1)^k {p\choose k} + (-1)^p\cdot 2 
=
\begin{cases}
1& \text{ if $p$ is even,}\\
-1 & \text{ if $p$  is odd}. 
\end{cases}
$$
In particular, \(K_{S_2}\neq 0\).

The same computation applies to \(S_3\). Indeed, replacing the last root
\(e_{2p-1}-e_{2p}\) by \(e_{2p-1}+e_{2p}\) gives the same Cartan classes and
the same lifting constants. Hence
\[
K_{S_3}=K_{S_2}=(-1)^p.
\]
Therefore both small representations
\[
J(\{1,2\},\ldots,\{2p-1,2p\},\{2p+1\}^*)
\]
and
\[
J(\{1,2\},\ldots,\{2p-3,2p-2\}, \{-(2p-1),-2p\},\{2p+1\}^*)
\]
occur in \(\LLift(\bbC)\).

Finally we prove (c). Let
$\tu G=\tu{\Spin}(2p+3,2p-1),$
and let
$\tu\delta=c_{S_1}(\tu\gamma_\sharp).$
As in the proof of (a), for each subset \(S'\subseteq S_1\), define
\[
\gamma_{S'}:=c_{S'}(\gamma_0).
\]
The relevant Cartan subgroups \(H_{S'}\) and lifting constants \(C(H_{S'})\)
are listed in Table \ref{tab:C-H-spin-2p3-2p1}.

\begin{table}[!htbp]
\centering
\renewcommand{\arraystretch}{1.35}
\[
\begin{array}{c|c|c|c}
S' & \#\text{ of such }S' & H_{S'} & C(H_{S'}) \\
\hline
\emptyset
&
1
&
(\bbR^\times)^{2p-2}\times S^1\times \bbC^\times
&
1
\\

\{\alpha_i\}
&
{p-1\choose 1}
&
(\bbR^\times)^{2p-4}\times S^1\times(\bbC^\times)^2
&
1
\\

\vdots
&
\vdots
&
\vdots
&
\vdots
\\

\{\alpha_{i_1},\ldots,\alpha_{i_k}\}
&
{p-1\choose k}
&
(\bbR^\times)^{2p-2-2k}\times S^1\times(\bbC^\times)^{k+1}
&
1
\\

\vdots
&
\vdots
&
\vdots
&
\vdots
\\

S_1-\{\alpha_i\}
&
{p-1\choose p-2}
&
(\bbR^\times)^2\times S^1\times(\bbC^\times)^{p-1}
&
1
\\

S_1
&
1
&
S^1\times(\bbC^\times)^p
&
1
\end{array}
\]
\caption{The Cartan subgroups \(H_{S'}\) and lifting constants \(C(H_{S'})\) for \(S'\subseteq S_1\) in the case \(G=\Spin(2p+3,2p-1)\).}
\label{tab:C-H-spin-2p3-2p1}
\end{table}

For \(|S'|=k\), we have 
$\ell(\gamma_0)-\ell(\gamma_{S'})=k,$
and Table \ref{tab:C-H-spin-2p3-2p1} gives
$C(H_{S'})=1$ for every $S'\subseteq S_1$. Therefore, as in (a), the coefficient of
\(J(\tu\delta)\) is
\[
K_{S_1}
=
\sum_{S'\subseteq S_1}
(-1)^{\ell(\gamma_0)-\ell(\gamma_{S'})}
C(H_{S'})
=
\sum_{k=0}^{p-1}
(-1)^k {p-1\choose k}
=
0,
\]
where we use \(p\geq 2\).

Thus
\[
J(\{1,2\},\ldots,\{2p-3,2p-2\},\{2p-1,2p+1\}^*)
\notin \LLift(\bbC).
\]
Since this is the only small representation in
$\mathcal R_{\rho/2}(\tu G)_\chi$
for \(\tu G=\tu{\Spin}(2p+3,2p-1)\), and since $\LLift(\bbC)_\chi\subseteq \mathcal R_{\rho/2}(\tu G)_\chi,$
we conclude that
\[
\LLift(\bbC)_\chi=\emptyset.
\]
\end{proof}

We now introduce notation for the representations in
$\prod_{\rho/2}^s(\tu G)$ that will be used throughout the remainder of the paper. Some of these representations have been shown to occur in $\LLift(\bbC)$.

\begin{notation}\label{n:small}

\begin{itemize}
    \item[(a)] Suppose
    \(\widetilde G=\widetilde{\Spin}(2p+1,2p-1)\).
    Define
    \begin{align*}
    \Omega_2
    &=
    J(\{2p\}^*)_{\chi_2},\\
    \Phi_1
    &=
    J(\{1,2\},\ldots,\{2p-3,2p-2\},
      \{2p-1\}^*)_{\chi_1}.
    \end{align*}

    \item[(b)] Suppose
    \(\widetilde G=\widetilde{\Spin}(2p+2,2p)\).
    For
    \(\chi\in\widehat Z_{\rho/2}(\widetilde G) =\{\chi_2,\chi_4\}\), define
    \begin{align*}
        \Omega_\chi
        &=
        J(\{2p+1\}^*)_\chi,\\
        \Phi_\chi
        &=
        J(\{1,2\},\ldots,\{2p-1,2p\},
          \{2p+1\}^*)_\chi,\\
        \Xi_\chi
        &=
        J(\{1,2\},\ldots,\{2p-3,2p-2\},
          \{-(2p-1),-2p\},\{2p+1\}^*)_\chi.
    \end{align*}
  For simplicity, 
    we also write
    \[
    \Omega_i=\Omega_{\chi_i},
    \qquad
    \Phi_i=\Phi_{\chi_i},
    \qquad
    \Xi_i=\Xi_{\chi_i},
    \qquad i=2,4.
    \]

    \item[(c)] Suppose
    \(\widetilde G=\widetilde{\Spin}(2p+3,2p-1)\).
    Since
    \(\widehat Z_{\rho/2}(\widetilde G)=\{\chi_2\}\), define
    \[
    \Phi_2
    =
    J(\{1,2\},\ldots,\{2p-3,2p-2\},
      \{2p-1,2p+1\}^*)_{\chi_2}.
    \]
\end{itemize}

\end{notation}

Using Notation \ref{n:small}, the following corollary follows immediately from Theorem \ref{t:main} together
with the coefficient computations in its proof.

\begin{cor}
Let
$\widetilde G\in
\{
\widetilde{\Spin}(2p+1,2p-1),
\widetilde{\Spin}(2p+2,2p),
\widetilde{\Spin}(2p+3,2p-1)
\}.$

Using Notation~\ref{n:small}, we have:

\begin{itemize}
    \item[(a)] If
    \(\widetilde G=\widetilde{\Spin}(2p+1,2p-1)\), then
    \[
    \LiftC=\Omega_2.
    \]

    \item[(b)] If
    \(\widetilde G=\widetilde{\Spin}(2p+2,2p)\), then
    \[
    \LiftC
    =
    \Omega_2+\Omega_4
    +
    (-1)^p
    (\Phi_2+\Phi_4+\Xi_2+\Xi_4).
    \]

    \item[(c)] If
    \(\widetilde G=\widetilde{\Spin}(2p+3,2p-1)\), then
    \[
    \LiftC=0.
    \]
\end{itemize}
\end{cor}

\section{A Low-Rank Example: $\tu\Spin(5,3)$} \label{s:low-rank}

In this section, we illustrate the lifting calculation in the first
substantial low-rank case
\[
\tu G=\tu{\Spin}(5,3).
\]
Besides providing a direct verification of the general lifting formula,
this example exhibits the formation of stable standard sums, the role of
the lifting constants, and the cancellation arising from the nonlinear
Kazhdan--Lusztig multiplicities. Since it already contains the main
features of the general argument, we omit analogous calculations for the
other low-rank groups.

\subsection{Linear $\rho$-regular characters}

For the linear group $G=\Spin(5,3)$, there are three conjugacy classes of
Cartan subgroups. Using the notation of Notation~\ref{n:linear-param}, we
parametrize the \(\rho\)-regular characters attached to each Cartan subgroup
as follows:
\[
\renewcommand{\arraystretch}{1.6}
\begin{array}{c@{\qquad}c}
\toprule
\text{Cartan subgroup}
&
\rho\text{-regular characters}
\\
\midrule
H_0\cong(S^1)^2\times\bbC^\times
&
\gamma\bigl(\{\{i,j\},k\}^*\bigr)
\\[1mm]
H_1\cong(\bbC^\times)^2
&
\gamma\bigl(\{\eta i,\eta j\},\{k\}^*\bigr),
\qquad \eta=\pm1
\\[1mm]
H_2\cong\bbC^\times\times(\bbR^\times)^2
&
\gamma\bigl(\{i\}^*;\varepsilon\bigr)
\\
\bottomrule
\end{array}
\]

The notation has the following meaning:
\begin{itemize}
\item The indices \(i,j,k\) are distinct elements of
\(\{1,2,3,4\}\).

\item For \(H_0\), the distinguished pair \(\{i,j\}\) inside
\(\{\{i,j\},k\}^*\) indicates that the roots
$e_i\pm e_j$
are compact imaginary, whereas
$e_i\pm e_k$ and $e_j\pm e_k$
are noncompact imaginary. Thus, for a fixed triple
\(\{i,j,k\}\), the three choices of the distinguished pair give the
three possible arrangements of compact and noncompact imaginary roots.

\item For \(H_1\),
\[
\gamma\bigl(\{i,j\},\{k\}^*\bigr)
\]
means that \(e_i-e_j\) is noncompact imaginary and \(e_i+e_j\) is real,
whereas
\[
\gamma\bigl(\{-i,-j\},\{k\}^*\bigr)
\]
means that \(e_i+e_j\) is noncompact imaginary and \(e_i-e_j\) is real.

\item For \(H_2\),
\[
\varepsilon=(\varepsilon_1,\varepsilon_2,\varepsilon_3)
\]
is the sequence of signs assigned to the three indices not contained in
the starred singleton \(\{i\}^*\), arranged in decreasing order.
\end{itemize}

\subsection{Lifting of the trivial representation} We now consider the character formula of the trivial representation
\(\bbC\). Using the labeling from \atlas, we write
\begin{equation}\label{e:trivial-char-formula-53}
\bbC
=
J(\gamma_{45})
=
\sum_i \sgn(\gamma_i)I(\gamma_i).
\end{equation}
The terms in \eqref{e:trivial-char-formula-53} are listed in
Tables~\ref{tab:Spin53-H0} and~\ref{tab:Spin53-H1-H2}. For each
parameter, we record its Cartan subgroup, the lifting constant \(C(H)\),
its coefficient in the character formula, the corresponding stable
standard sum, and the \(\rho/2\)-regular genuine parameter obtained
after lifting. For \(H_1\) and \(H_2\), each standard module appearing
in the table is stable by itself.

\begin{table}[H]
\centering
\small
\renewcommand{\arraystretch}{1.35}
\setlength{\tabcolsep}{4pt}

\resizebox{\textwidth}{!}{%
\begin{tabular}{
@{}
c
c
c
c
c
c
@{}
}
\toprule
\text{Cartan}
&
\(C(H)\)
&
\(\sgn(\gamma_i)\)
&
\(\rho\)-regular parameters
&
\text{Stable standard sum}
&
\text{Lift}
\\
\midrule

\(H_0\)
&
\(2\)
&
\(-1\)
&
\(\begin{aligned}[t]
\gamma_{12}
&=
\gamma\bigl(\{\{1,4\},3\}^{*}\bigr),\\
\gamma_{13}
&=
\gamma\bigl(\{\{3,4\},1\}^{*}\bigr),\\
\gamma_{14}
&=
\gamma\bigl(\{\{1,3\},4\}^{*}\bigr)
\end{aligned}\)
&
\(\displaystyle
I(\gamma_{12})+I(\gamma_{13})+I(\gamma_{14})
\)
&
\(\displaystyle
\tu\gamma\bigl(\{1,3,4\}^{*}\bigr)
\)
\\[3mm]

\(H_0\)
&
\(2\)
&
\(1\)
&
\(\begin{aligned}[t]
\gamma_{21}
&=
\gamma\bigl(\{\{2,4\},3\}^{*}\bigr),\\
\gamma_{22}
&=
\gamma\bigl(\{\{3,4\},2\}^{*}\bigr),\\
\gamma_{23}
&=
\gamma\bigl(\{\{2,3\},4\}^{*}\bigr)
\end{aligned}\)
&
\(\displaystyle
I(\gamma_{21})+I(\gamma_{22})+I(\gamma_{23})
\)
&
\(\displaystyle
\tu\gamma\bigl(\{2,3,4\}^{*}\bigr)
\)
\\

\bottomrule
\end{tabular}%
}

\caption{The \(H_0\)-terms in the standard character formula of
\(J(\gamma_{45})=\bbC\).}
\label{tab:Spin53-H0}
\end{table}

\begin{longtable}{
@{}
c
c
c
>{\raggedright\arraybackslash}p{0.42\textwidth}
>{\raggedright\arraybackslash}p{0.30\textwidth}
@{}
}
\caption{The \(H_1\)- and \(H_2\)-terms in the standard character formula
of \(J(\gamma_{45})=\bbC\).}
\label{tab:Spin53-H1-H2}
\\
\toprule
Cartan
&
\(C(H)\)
&
\(\sgn(\gamma_i)\)
&
\(\rho\)-regular parameter
&
Lift
\\
\midrule
\endfirsthead

\toprule
Cartan
&
\(C(H)\)
&
\(\sgn(\gamma_i)\)
&
\(\rho\)-regular parameter
&
Lift
\\
\midrule
\endhead

\midrule
\multicolumn{5}{r}{\textit{Continued on the next page}}
\\
\endfoot

\bottomrule
\endlastfoot

\(H_1\)
&
\(1\)
&
\(-1\)
&
\(\gamma_9=\gamma\bigl(\{-2,-4\},\{1\}^*\bigr)\)
&
\(0\)
\\

\(H_1\)
&
\(1\)
&
\(-1\)
&
\(\gamma_{10}=\gamma\bigl(\{-1,-3\},\{2\}^*\bigr)\)
&
\(0\)
\\
\addlinespace

\(H_1\)
&
\(1\)
&
\(1\)
&
\(\gamma_{15}=\gamma\bigl(\{3,4\},\{1\}^*\bigr)\)
&
\(\tu\gamma\bigl(\{3,4\},\{1\}^*\bigr)\)
\\

\(H_1\)
&
\(1\)
&
\(1\)
&
\(\gamma_{16}=\gamma\bigl(\{-1,-3\},\{4\}^*\bigr)\)
&
\(0\)
\\

\(H_1\)
&
\(1\)
&
\(1\)
&
\(\gamma_{17}=\gamma\bigl(\{2,3\},\{1\}^*\bigr)\)
&
\(\tu\gamma\bigl(\{2,3\},\{1\}^*\bigr)\)
\\

\(H_1\)
&
\(1\)
&
\(1\)
&
\(\gamma_{18}=\gamma\bigl(\{-3,-4\},\{1\}^*\bigr)\)
&
\(\tu\gamma\bigl(\{-3,-4\},\{1\}^*\bigr)\)
\\

\(H_1\)
&
\(1\)
&
\(1\)
&
\(\gamma_{19}=\gamma\bigl(\{1,4\},\{2\}^*\bigr)\)
&
\(\tu\gamma\bigl(\{1,4\},\{2\}^*\bigr)\)
\\

\(H_1\)
&
\(1\)
&
\(1\)
&
\(\gamma_{20}=\gamma\bigl(\{-1,-4\},\{2\}^*\bigr)\)
&
\(\tu\gamma\bigl(\{-1,-4\},\{2\}^*\bigr)\)
\\
\addlinespace

\(H_1\)
&
\(1\)
&
\(-1\)
&
\(\gamma_{27}=\gamma\bigl(\{1,4\},\{3\}^*\bigr)\)
&
\(\tu\gamma\bigl(\{1,4\},\{3\}^*\bigr)\)
\\

\(H_1\)
&
\(1\)
&
\(-1\)
&
\(\gamma_{28}=\gamma\bigl(\{-1,-4\},\{3\}^*\bigr)\)
&
\(\tu\gamma\bigl(\{-1,-4\},\{3\}^*\bigr)\)
\\

\(H_1\)
&
\(1\)
&
\(-1\)
&
\(\gamma_{29}=\gamma\bigl(\{3,4\},\{2\}^*\bigr)\)
&
\(\tu\gamma\bigl(\{3,4\},\{2\}^*\bigr)\)
\\

\(H_1\)
&
\(1\)
&
\(-1\)
&
\(\gamma_{30}=\gamma\bigl(\{-2,-3\},\{4\}^*\bigr)\)
&
\(\tu\gamma\bigl(\{-2,-3\},\{4\}^*\bigr)\)
\\

\(H_1\)
&
\(1\)
&
\(-1\)
&
\(\gamma_{31}=\gamma\bigl(\{1,3\},\{2\}^*\bigr)\)
&
\(0\)
\\

\(H_1\)
&
\(1\)
&
\(-1\)
&
\(\gamma_{32}=\gamma\bigl(\{-3,-4\},\{2\}^*\bigr)\)
&
\(\tu\gamma\bigl(\{-3,-4\},\{2\}^*\bigr)\)
\\
\addlinespace

\(H_1\)
&
\(1\)
&
\(1\)
&
\(\gamma_{36}=\gamma\bigl(\{1,2\},\{3\}^*\bigr)\)
&
\(\tu\gamma\bigl(\{1,2\},\{3\}^*\bigr)\)
\\

\(H_1\)
&
\(1\)
&
\(1\)
&
\(\gamma_{37}=\gamma\bigl(\{1,3\},\{4\}^*\bigr)\)
&
\(0\)
\\

\(H_1\)
&
\(1\)
&
\(1\)
&
\(\gamma_{38}=\gamma\bigl(\{2,4\},\{3\}^*\bigr)\)
&
\(0\)
\\

\(H_1\)
&
\(1\)
&
\(1\)
&
\(\gamma_{39}=\gamma\bigl(\{-2,-4\},\{3\}^*\bigr)\)
&
\(0\)
\\
\addlinespace

\(H_1\)
&
\(1\)
&
\(-1\)
&
\(\gamma_{43}=\gamma\bigl(\{1,2\},\{4\}^*\bigr)\)
&
\(\tu\gamma\bigl(\{1,2\},\{4\}^*\bigr)\)
\\

\midrule

\(H_2\)
&
\(1\)
&
\(-1\)
&
\(\gamma_{26}=\gamma\bigl(\{1\}^*;(+,+,+)\bigr)\)
&
\(\tu\gamma\bigl(\{1\}^*\bigr)\)
\\

\(H_2\)
&
\(1\)
&
\(-1\)
&
\(\gamma_{40}=\gamma\bigl(\{3\}^*;(+,+,+)\bigr)\)
&
\(\tu\gamma\bigl(\{3\}^*\bigr)\)
\\

\(H_2\)
&
\(1\)
&
\(1\)
&
\(\gamma_{33}=\gamma\bigl(\{2\}^*;(+,+,+)\bigr)\)
&
\(\tu\gamma\bigl(\{2\}^*\bigr)\)
\\

\(H_2\)
&
\(1\)
&
\(1\)
&
\(\gamma_{45}=\gamma\bigl(\{4\}^*;(+,+,+)\bigr)\)
&
\(\tu\gamma\bigl(\{4\}^*\bigr)\)
\\

\end{longtable}

After grouping the \(H_0\)-terms into stable standard sums and applying
the lifting formula, we obtain
\begin{equation}\label{e:lift-triv-53}
\LiftC
=
\sum_{\tu\gamma}
a_{\tu\gamma}I(\tu\gamma),
\end{equation}
where \(a_{\tu\gamma}\) is the sum of the quantities
\[
\sgn(\gamma_i)C(H_i)
\]
over the stable standard sums whose lifts have parameter
\(\tu\gamma\). The nonzero parameters and the corresponding coefficients
are determined by Tables~\ref{tab:Spin53-H0}
and~\ref{tab:Spin53-H1-H2}.

The following lemma is a special low-rank case of
Lemma~\ref{l:comp-factor}.

\begin{lemma}\label{l:53-m-coeff}
Consider $\widetilde G=\widetilde{\Spin}(5,3),$
and set
\[
\widetilde\gamma_0
=
\widetilde\gamma(\{3\}^*)_{\chi_1},
\qquad
\widetilde\gamma_1
=
\widetilde\gamma(\{1,2\},\{3\}^*)_{\chi_1}.
\]
Then
\[
m(\widetilde\gamma_1,\widetilde\gamma_1)
=
m(\widetilde\gamma_1,\widetilde\gamma_0)
=
1.
\]
Moreover, among the standard modules appearing in
\eqref{e:lift-triv-53}, only
$I(\widetilde\gamma_1)$ and 
$I(\widetilde\gamma_0)$
contain \(J(\widetilde\gamma_1)\).
\end{lemma}

\begin{proof}
By the length calculation, the only parameters occurring in
\eqref{e:lift-triv-53} that can possibly contribute to
\(J(\widetilde\gamma_1)\) are
\[
\widetilde\gamma_1,\quad
\widetilde\gamma_0,\quad
\widetilde\gamma_2,\quad
\widetilde\gamma_3,\quad
\widetilde\gamma_4,
\]
where
\[
\widetilde\gamma_2
=
\widetilde\gamma(\{1,2\},\{4\}^*)_{\chi_2},
\qquad
\widetilde\gamma_3
=
\widetilde\gamma(\{2,3\},\{4\}^*)_{\chi_2},
\qquad
\widetilde\gamma_4
=
\widetilde\gamma(\{4\}^*)_{\chi_2}.
\]
The parameters
\(\widetilde\gamma_0\) and \(\widetilde\gamma_1\) have genuine
central character \(\chi_1\), whereas
\(\widetilde\gamma_2,\widetilde\gamma_3,\widetilde\gamma_4\)
have genuine central character \(\chi_2\).
Since every irreducible constituent of a standard module has the
same genuine central character as the standard module, it follows
immediately that
\[
m(\widetilde\gamma_1,\widetilde\gamma_i)=0,
\qquad i=2,3,4.
\]
Thus only
\(I(\widetilde\gamma_1)\) and \(I(\widetilde\gamma_0)\)
can contain \(J(\widetilde\gamma_1)\).

Clearly,
$m(\widetilde\gamma_1,\widetilde\gamma_1)=1.$
It remains to compute
\(m(\widetilde\gamma_1,\widetilde\gamma_0)\).

Let $s=s_{e_3-e_4},$
the reflection through the nonintegral root \(e_3-e_4\).
By \cite[Proposition 7.10, Case II (iv)]{RT00},
\[
\begin{aligned}
M(\widetilde\gamma_1,\widetilde\gamma_0)
&=
-P_{\widetilde\gamma_1,\widetilde\gamma_0}(1)\\
&=
-P_{s\times\widetilde\gamma_1,
   (s\times\widetilde\gamma_0)_\beta}(1)\\
&=
-P_{s\times\widetilde\gamma_1,
   s\times\widetilde\gamma_1}(1)
=-1,
\end{aligned}
\]
where \(\beta=e_1-e_2\).

Let
\[
\mathbf M
=
\bigl(M(\widetilde\gamma,\widetilde\delta)\bigr),
\qquad
\mathbf m
=
\bigl(m(\widetilde\gamma,\widetilde\delta)\bigr).
\]
Since \(\mathbf M\mathbf m=I\), the
\((\widetilde\gamma_1,\widetilde\gamma_0)\)-entry gives
\[
0
=
M(\widetilde\gamma_1,\widetilde\gamma_1)
m(\widetilde\gamma_1,\widetilde\gamma_0)
+
M(\widetilde\gamma_1,\widetilde\gamma_0)
m(\widetilde\gamma_0,\widetilde\gamma_0).
\]
Since
$M(\widetilde\gamma_1,\widetilde\gamma_1)
=
m(\widetilde\gamma_0,\widetilde\gamma_0)
=1$
and
$M(\widetilde\gamma_1,\widetilde\gamma_0)=-1,$
we obtain
\[
m(\widetilde\gamma_1,\widetilde\gamma_0)=1.
\]
\end{proof}

\begin{prop}\label{p:lift-triv-53}
Let
$G=\Spin(5,3),$ and 
$\tu G=\tu{\Spin}(5,3),$
and let \(\chi_1,\chi_2\) be the two \(\rho/2\)-compatible genuine central characters, as in
Proposition~\ref{prop:centers-compatible-central-characters}.
Then
\[
\LiftC
=
J(\{4\}^*)_{\chi_2}
=
\Omega_2.
\]
\end{prop}
\begin{proof}
By Theorem~\ref{t:lift-trivial-small}, the only possible irreducible
constituents of \(\LiftC\) are
\[
J(\{4\}^*)_{\chi_2}=\Omega_2
\qquad\text{and}\qquad
J(\{1,2\},\{3\}^*)_{\chi_1}=\Phi_1.
\]

We first compute the coefficient of
\(J(\{4\}^*)_{\chi_2}\) in \eqref{e:lift-triv-53}.
By triangularity, the only standard module occurring in
\eqref{e:lift-triv-53} that can contain
\(J(\{4\}^*)_{\chi_2}\) is $I(\{4\}^*)_{\chi_2}.$
Since this standard module occurs with coefficient \(1\),
\(J(\{4\}^*)_{\chi_2}\) occurs in \(\LiftC\) with coefficient \(1\).

Next consider
$J(\{1,2\},\{3\}^*)_{\chi_1}.$
By Lemma~\ref{l:53-m-coeff}, among the standard modules occurring in
\eqref{e:lift-triv-53}, precisely
\[
I(\{1,2\},\{3\}^*)_{\chi_1}
\qquad\text{and}\qquad
I(\{3\}^*)_{\chi_1}
\]
contain this irreducible representation, in each case with
multiplicity one. By Table~\ref{tab:Spin53-H1-H2}, their coefficients
in \eqref{e:lift-triv-53} are respectively
\[
C(H_1)
\qquad\text{and}\qquad
-C(H_2).
\]
Since
$C(H_1)=C(H_2)=1,$
the coefficient of \(J(\{1,2\},\{3\}^*)_{\chi_1}\) is
\[
C(H_1)-C(H_2)=0.
\]
Consequently,
\[
\LiftC
=
J(\{4\}^*)_{\chi_2}
=
\Omega_2.
\]
\end{proof}

\section{$\tu K$-types, Associated Varieties and Unitarity}

In \cite{BTs18}, the $\tu K$-types of certain unipotent representations
of nonlinear real Spin groups were computed.  In this final section,
included mainly for completeness and reference, we identify the small
representations studied in this paper with those appearing in
\cite{BTs18} and translate the corresponding $\tu K$-type formulas
into our notation.  We then recall their real associated varieties
and unitarity properties.

\subsection{\(\tu K\)-spectra of small representations}

The complete $\tu K$-spectra of the representations considered here
can be obtained from \cite[Theorem~4.1]{BTs18}.  We translate the
formulas there into the notation of the present paper.

For $\tu G=\tu{\Spin}(2p+2,2p)$, we first record the lowest
$\tu K$-types in order to identify our representations
$\Omega_i,\Phi_i,\Xi_i$ with the corresponding representations in
\cite[Theorem~4.1]{BTs18}.

\begin{prop}\label{p:lowest-K-types-even-even}
Let
\[
\tu G=\tu{\Spin}(2p+2,2p),
\qquad
\tu K=\Spin(2p+2)\times\Spin(2p),
\]
where \(p\geq 1\). Let \(\chi_2\) and \(\chi_4\) be the two $\rho/2$-compatible  genuine 
central characters (see Proposition \ref{prop:centers-compatible-central-characters}), and write
\[
\Omega_i=\Omega_{\chi_i},
\qquad
\Phi_i=\Phi_{\chi_i},
\qquad
\Xi_i=\Xi_{\chi_i},
\qquad i\in\{2,4\}.
\]

Set \(\varepsilon_p=(-1)^p\). 
The lowest
\(\tu K\)-types are given by the following table.  The \(\Omega\)- and \(\Phi\)-rows in the
following table are valid for \(p\geq1\), while the \(\Xi\)-row is valid
for \(p\geq2\). 
\[
\renewcommand{\arraystretch}{1.8}
\begin{array}{c@{\qquad}c@{\qquad}c}
\toprule
&
\text{central character }\chi_2
&
\text{central character }\chi_4
\\
\midrule
\Omega
&
\displaystyle
\left(
0^{p+1}
\ \middle|\
\left(\frac12\right)^p
\right)
&
\displaystyle
\left(
0^{p+1}
\ \middle|\
\left(\frac12\right)^{p-1},-\frac12
\right)
\\[2mm]
\Phi
&
\displaystyle
\left(
0^{p+1}
\ \middle|\
2-\frac{\varepsilon_p}{2},
\left(\frac32\right)^{p-1}
\right)
&
\displaystyle
\left(
0^{p+1}
\ \middle|\
2+\frac{\varepsilon_p}{2},
\left(\frac32\right)^{p-1}
\right)
\\[2mm]
\Xi
&
\displaystyle
\left(
0^{p+1}
\ \middle|\
2+\frac{\varepsilon_p}{2},
\left(\frac32\right)^{p-2},
-\frac32
\right)
&
\displaystyle
\left(
0^{p+1}
\ \middle|\
2-\frac{\varepsilon_p}{2},
\left(\frac32\right)^{p-2},
-\frac32
\right)
\\
\bottomrule
\end{array}
\]
For \(p=1\), the lowest \(\tu K\)-types of the \(\Xi\)-representations are
\[
\Xi_2:\quad
\left(0^2\mid-\frac32\right),
\qquad
\Xi_4:\quad
\left(0^2\mid-\frac52\right).
\]

Here the entry in the \(\chi_i\)-column and the \(\Omega\), \(\Phi\), or
\(\Xi\) row denotes the lowest \(\tu K\)-type of
\(\Omega_i\), \(\Phi_i\), or \(\Xi_i\), respectively.

\begin{proof}
Fix the standard positive compact root system for
$\tu K=\Spin(2p+2)\times\Spin(2p).$
The result follows by applying the minimal \(K\)-type formula of
\cite{Kn86} to the Langlands parameters defining
the $\Omega, \Phi, \Xi$. 
The formula applies to the present nonlinear double cover in the same way as
in the split case, since the calculation depends only on the root data and
the lowest \(\tu K\cap\tu M\)-types of the inducing representations.
Substitution of the six Langlands parameters into the formula gives the
highest weights displayed in the table.
\end{proof}

\end{prop}

\begin{thm}\label{thm:K-types-even-even}
Let
\[
\tu G=\tu{\Spin}(2p+2,2p),
\qquad
\tu K=\Spin(2p+2)\times\Spin(2p),
\]
where \(p\geq 1\). We parametrize an irreducible \(\tu K\)-type by its
highest weight
\[
(\lambda_1,\ldots,\lambda_{p+1}
 \mid
 \mu_1,\ldots,\mu_p).
\]
Then the \(\tu K\)-types of the six representations
\[
\Omega_2,\quad \Omega_4,\quad
\Phi_2,\quad \Phi_4,\quad
\Xi_2,\quad \Xi_4
\]
are described as follows.

\begin{enumerate}
\item The \(\tu K\)-types of  \(\Omega_2\) and \(\Omega_4\) are of the form
\begin{equation}\label{e:2p2-2p-K}
\left(
\beta_1,\ldots,\beta_p,0
\ \middle|\
\delta_1,\ldots,\delta_p
\right),
\end{equation}
where
\[
\beta_1\geq\cdots\geq\beta_p\geq0,
\qquad
\beta_j\in\bbZ,
\qquad
\delta_j\in\bbZ+\frac12,
\]
and
\[
\beta_1+\frac12
\geq\delta_1
\geq\beta_2+\frac12
\geq\delta_2
\geq\cdots
\geq\beta_p+\frac12
\geq|\delta_p|.
\]
Moreover, the \(\tu K\)-type in \eqref{e:2p2-2p-K} occurs in
\(\Omega_2\) if and only if
\[
\sum_{j=1}^p
\left(\beta_j+\delta_j-\frac12\right)
\in 2\bbZ,
\]
and it occurs in \(\Omega_4\) if and only if
\[
\sum_{j=1}^p
\left(\beta_j+\delta_j-\frac12\right)
\in 2\bbZ+1.
\]

\item 
The \(\tu K\)-types of \(\Phi_2\)  and $\Phi_4$  are of the form
\begin{equation}\label{e:2p2-2p-Phi}
\left(
\delta_1,\ldots,\delta_p,0
\ \middle|\
\beta_1+\frac32,\ldots,\beta_p+\frac32
\right),
\end{equation}
where \(\beta_j,\delta_j\in\bbZ\) satisfy
\[
\beta_1
\geq\delta_1
\geq\beta_2
\geq\delta_2
\geq\cdots
\geq\beta_p
\geq\delta_p
\geq0.
\]
Moreover, the \(\tu K\)-type in \eqref{e:2p2-2p-Phi} occurs in
\(\Phi_2\) if and only if
\[
\sum_{j=1}^p
\left(\beta_j+\delta_j-1\right)
\in 2\bbZ,
\]
and it occurs in \(\Phi_4\) if and only if
\[
\sum_{j=1}^p
\left(\beta_j+\delta_j-1\right)
\in 2\bbZ+1.
\]

\item The \(\tu K\)-types of \(\Xi_2\) and $\Xi_4$ are of the form
\begin{equation} \label{e:2p2-2p-Xi}
    \left(
\delta_1,\ldots,\delta_p,0
\ \middle|\
\beta_1+\frac32,\ldots,\beta_{p-1}+\frac32,
-\left(\beta_p+\frac32\right)
\right),
\end{equation}

where \(\beta_j,\delta_j\in\bbZ\) satisfy
\[
\beta_1
\geq\delta_1
\geq\beta_2
\geq\delta_2
\geq\cdots
\geq\beta_p
\geq\delta_p
\geq0.
\]

Moreover, the \(\tu K\)-type in \eqref{e:2p2-2p-Xi} occurs in
\(\Xi_2\) if and only if
\[
\sum_{j=1}^p
\left(\beta_j+\delta_j-1\right)
\in 2\bbZ+1,
\]
and it occurs in \(\Xi_4\) if and only if
\[
\sum_{j=1}^p
\left(\beta_j+\delta_j-1\right)
\in 2\bbZ.
\]

\end{enumerate}

Each of the above \(\tu K\)-types occurs with multiplicity one.
\begin{proof}

In the notation of \cite[Theorem~4.1, Case~2]{BTs18}, we take
\[
P=p+1,\qquad Q=p,\qquad r_+=P-Q=1.
\]
The three pairs of representations denoted there by
\[
\tau_i,\qquad \pi_i,\qquad \sigma_i
\]
are identified, respectively, with
\[
\Omega_i,\qquad \Phi_i,\qquad \Xi_i,
\qquad i\in\{2,4\},
\]
by comparing their genuine central characters and their lowest
\(\tu K\)-types, as computed in Proposition
\ref{p:lowest-K-types-even-even}.

    The \(\tu K\)-type formulas for the representations \(\tau_i\) are obtained in
\cite{BTs18} by restricting the small representations of
\(\tu{\Spin}(2p+2,2p+1)\) studied in \cite{LS08} to
\(\tu{\Spin}(2p+2,2p)\), whereas the formulas for the representations
\(\pi_i\) and \(\sigma_i\) are obtained by restricting the corresponding
representations of \(\tu{\Spin}(2p+3,2p)\).
Translating the parity conditions into our labeling of the central characters
\(\chi_2\) and \(\chi_4\), and incorporating the corrections recorded in
Remark~\ref{r:BTs18-corrections}, gives the formulas stated above.
\end{proof}

\end{thm}
\begin{remark}\label{r:BTs18-corrections}
Our labeling is normalized so that the representation with lowest
\(\tu K\)-type
\[
\left(
0^{p+1}
\ \middle|\
\left(\frac12\right)^p
\right)
\]
has central character \(\chi_2\), while the representation with lowest
\(\tu K\)-type
\[
\left(
0^{p+1}
\ \middle|\
\left(\frac12\right)^{p-1},-\frac12
\right)
\]
has central character \(\chi_4\). Accordingly, the parity conditions in
\cite[Theorem~4.1, Case~2]{BTs18} are relabeled to agree with this convention.
We have also corrected the corresponding parity assignments for the
\(\Phi\)- and \(\Xi\)-families.
\end{remark}

\begin{comment}

\begin{prop}\label{p:lowest-K-types-2p1}
Let
\[
\tu G=\tu{\Spin}(2p+1,2p-1),
\qquad
\tu K=\Spin(2p+1)\times\Spin(2p-1),
\]
where \(p\geq 2\). Let \(\chi_1\) and \(\chi_2\) be the two $\rho/2$-compatible  genuine 
central characters (see Proposition \ref{prop:centers-compatible-central-characters}), and write
\[
\Omega_2=\Omega_{\chi_2},
\qquad
\Phi_1=\Phi_{\chi_1}.
\]

The lowest
\(\tu K\)-types of these representations are given by
\[
\renewcommand{\arraystretch}{1.8}
\begin{array}{c@{\qquad}c@{\qquad}c}
\toprule
&
\text{central character }\chi_1
&
\text{central character }\chi_2
\\
\midrule
\Omega
&&
\displaystyle
\left(
0^{p}
\ \middle|\
\left(\frac12\right)^{p-1} 
\right)
\\[2mm]
\Phi
&
\displaystyle
\left(
\left(\frac12\right)^p
\ \middle|\
1^{p-1}
\right)
&
\\
\bottomrule
\end{array}
\]
The empty entries reflect the fact that, with the notation of
Notation~\ref{n:small}, only \(\Phi_1\) and \(\Omega_2\) occur for
this group.
\end{prop}
\end{comment}

\begin{thm}\label{thm:K-types-2p1}
Let
\[
\tu G=\tu{\Spin}(2p+1,2p-1),
\qquad
\tu K=\Spin(2p+1)\times\Spin(2p-1),
\]
where \(p\geq2\). We parametrize an irreducible \(\tu K\)-type by its
highest weight
\[
(\beta_1,\ldots,\beta_p
 \mid
 \delta_1,\ldots,\delta_{p-1}).
\]

Then the \(\tu K\)-types of the two small representations
\(\Phi_1\) and \(\Omega_2\) are described as follows.

\begin{enumerate}

\item
The \(\tu K\)-types of \(\Phi_1\) are precisely those of the form
\begin{equation}\label{e:2p1-2p1-chi1}
\left(
\beta_1,\ldots,\beta_p
\ \middle|\
\delta_1,\ldots,\delta_{p-1}
\right),
\end{equation}
where
\[
\beta_j\in\bbZ+\frac12,
\qquad
\delta_j\in\bbZ,
\qquad
\beta_p\geq0,
\]
and
\[
\beta_1+\frac12
\geq\delta_1
\geq\beta_2+\frac12
\geq\delta_2
\geq\cdots
\geq\beta_{p-1}+\frac12
\geq\delta_{p-1}
\geq\beta_p+\frac12.
\]

\item
The \(\tu K\)-types of \(\Omega_2\) are precisely those of the form
\begin{equation}\label{e:2p1-2p1-chi2}
\left(
\beta_1,\ldots,\beta_p
\ \middle|\
\delta_1,\ldots,\delta_{p-1}
\right),
\end{equation}
where
\[
\beta_j\in\bbZ,
\qquad
\delta_j\in\bbZ+\frac12,
\qquad
\beta_p\geq0,
\]
and
\[
\beta_1+\frac12
\geq\delta_1
\geq\beta_2+\frac12
\geq\delta_2
\geq\cdots
\geq\beta_{p-1}+\frac12
\geq\delta_{p-1}
\geq\beta_p+\frac12.
\]

\end{enumerate}
\begin{proof}

This is Case~5 of \cite[Theorem~4.1]{BTs18}, with $q=p$ and hence
$r_+=p-q=0$.  The two representations occurring there are denoted
by $\pi_1$ and $\pi_2$.  Their $\tu K$-types have, respectively, the
integrality patterns
\[
\bbZ\mid\left(\bbZ+\frac12\right),
\qquad
\left(\bbZ+\frac12\right)\mid\bbZ.
\]
By Lemma~\ref{l:central-char}, these correspond to the genuine central
characters $\chi_2$ and $\chi_1$, respectively.  Since
$\Omega_2$ and $\Phi_1$ are the unique small representations with
these central characters, we obtain
\[
\pi_1=\Omega_2,\qquad \pi_2=\Phi_1.
\]
The stated $\tu K$-type formulas then follow from
\cite[Theorem~4.1, Case~5]{BTs18}.
\end{proof}

\end{thm}

\begin{comment}
    
 \begin{remark}\label{r:BTs18-corrections-2p1}
There is a correction to
\cite[Theorem~4.1, Case~5 with \(r_+=0\)]{BTs18}.
For
\[
\tu G=\tu{\Spin}(2p+1,2p-1),
\]
only two representations, denoted there by \(\pi_1\) and \(\pi_2\), are
listed, one for each \(\rho/2\)-compatible genuine central character.
In fact, each of the two \(\tu K\)-spectra displayed there splits into
two disjoint parts, distinguished by a parity condition. These four
parts are the \(\tu K\)-spectra of
\[
\Omega_1,\qquad \Phi_1,\qquad
\Omega_2,\qquad \Phi_2.
\]
More precisely,
\[
\left.\pi_i\right|_{\tu K}
=
\left.\Omega_i\right|_{\tu K}
\oplus
\left.\Phi_i\right|_{\tu K},
\qquad i\in\{1,2\},
\]
where the two summands are distinguished by the parity conditions in
Theorem~\ref{thm:K-types-2p1}.

Consequently, there are four irreducible representations in this case,
rather than the two listed in \cite[Theorem~4.1]{BTs18}.
In \cite{BTs18}, we also counted the number of small representations
associated with each \(\rho/2\)-compatible genuine central character.
The count given there differs from the one obtained in the present paper.
We refer the reader to Sections~\ref{s:central-char} and~\ref{s:counting}
for a revised version of that discussion.

Accordingly, the value \(n_{\cal O}=2\) in
\cite[Theorem~6.17, Cases~5 and~6]{BTs18}, for
\(\tu{\Spin}(2p+1,2p-1)\), should be replaced by
\[
n_{\cal O}=4.
\]
\end{remark}
\end{comment}

\begin{thm}\label{thm:K-types-2p3}
Let
\[
\tu G=\tu{\Spin}(2p+3,2p-1),
\qquad
\tu K=\Spin(2p+3)\times\Spin(2p-1),
\]
where \(p\geq2\). Let \(\Phi_2\) denote the unique small representation in
\(\prod_{\rho/2}^s(\tu G)\), as in Notation~\ref{n:small}.

We parametrize an irreducible \(\tu K\)-type by its highest weight
\[
(\lambda_1,\ldots,\lambda_{p+1}
 \mid
 \mu_1,\ldots,\mu_{p-1}).
\]
Then the \(\tu K\)-types of \(\Phi_2\) are of the form
\begin{equation}\label{e:2p3-2p1-K}
\left(
\beta_1,\ldots,\beta_p,0
\ \middle|\
\delta_1,\ldots,\delta_{p-1}
\right),
\end{equation}
where
\[
\beta_1\geq\cdots\geq\beta_p\geq0,
\qquad
\beta_j\in\bbZ,
\qquad
\delta_j\in\bbZ+\frac12,
\]
and
\[
\beta_1+\frac32
\geq\delta_1
\geq\beta_2+\frac32
\geq\delta_2
\geq\cdots
\geq\beta_{p-1}+\frac32
\geq\delta_{p-1}
\geq\beta_p+\frac32.
\]
\end{thm}

\begin{proof}
This is Case~5 of \cite[Theorem~4.1]{BTs18}, with
$P=p+1$, $Q=p$, and hence $r_+=1$.
In this case there is a single representation with the relevant
genuine central character. By the counting result in Section~\ref{s:counting} and the
central-character computation in Section~\ref{s:central-char}, the unique small
representation in the present paper is $\Phi_2$.  Hence the
$\tu K$-spectrum stated above is exactly the one given in
\cite[Theorem~4.1, Case~5]{BTs18}.
\end{proof}

\subsection{Real associated varieties}

Let \(\mathcal O_{\bbC}\) denote the complex nilpotent orbit associated with
the groups under consideration. Thus,
\begin{align*}
\mathcal O_{\bbC}
&=
[3\,2^{2p-2}\,1],
&&
\text{if }
\tu G=\tu{\Spin}(2p+1,2p-1),\\
\mathcal O_{\bbC}
&=
[3\,2^{2p-2}\,1^3],
&&
\text{if }
\tu G=\tu{\Spin}(2p+2,2p)
\text{ or }
\tu{\Spin}(2p+3,2p-1).
\end{align*}

We label the real forms \(\mathcal O_{\bbR}\) of \(\mathcal O_{\bbC}\)
by their signed Young diagrams; see \cite{CM93}. In particular, a real
orbit with underlying partition
\[
[3\,2^{2p-2}\,1^k]
\]
is written in the form
\[
[3^\epsilon\,2^{2p-2}\,1^{+,a}1^{-,b}],
\]
possibly with an additional label \(\mathrm{I}\) or \(\mathrm{II}\).
Here \(3^\epsilon\) denotes the row of length \(3\) beginning with the
sign \(\epsilon\), while \(1^{+,a}\) denotes \(a\) rows of length \(1\)
labeled \(+\); the exponent \(a\) is omitted when \(a=1\). The notation
\(1^{-,b}\) is defined similarly.

Based on \cite[Section~4.3]{BTs18}, the following tables record the small
representations according to the pairs
\[
(\mathcal O_{\bbR},\chi),
\]
where \(\mathcal O_{\bbR}\) is a real form of \(\mathcal O_{\bbC}\) and
\(\chi\) is a genuine central character compatible with \(\rho/2\).
An entry \(\pi\) in the \(\mathcal O_{\bbR}\)-column and the \(\chi\)-row
means that
\[
\operatorname{AV}_{\bbR}(\pi)
=
\overline{\mathcal O_{\bbR}}
\]
and that the central character of \(\pi\) is \(\chi\). When more than one
representation appears in an entry, all of them have the same real
associated variety and central character.

\subsubsection{The case \(\tu{\Spin}(2p+1,2p-1)\)}

the only real form of
\[
\mathcal O_{\bbC}=[3\,2^{2p-2}\,1]
\]
whose closure occurs as the real associated variety  of a small representation is
\[
\mathcal O_{\bbR}
=
[3^+\,2^{2p-2}\,1^+].
\]
In fact, both small representations have real associated variety
\(\overline{\mathcal O_{\bbR}}\).

\[
\renewcommand{\arraystretch}{1.6}
\begin{array}{c@{\qquad}c}
\toprule
&
\substack{
\mathcal O_{\bbR}\\
=[3^+\,2^{2p-2}\,1^+]
}
\\
\midrule
\chi_1
&
 \Phi_1
\\
\chi_2
&
\Omega_2
\\
\bottomrule
\end{array}
\]

\subsubsection{The case \(\tu{\Spin}(2p+2,2p)\)}

The real forms of
\[
\mathcal O_{\bbC}=[3\,2^{2p-2}\,1^3]
\]
whose closures occur as the real associated varieties of small representations are
\[
\begin{aligned}
\mathcal O_{\bbR,1}
&=
[3^+\,2^{2p-2}\,1^-\,1^{+,2}],\\
\mathcal O_{\bbR,2}
&=
[3^-\,2^{2p-2}\,1^{+,3}]_{\mathrm{I}},\\
\mathcal O_{\bbR,3}
&=
[3^-\,2^{2p-2}\,1^{+,3}]_{\mathrm{II}}.
\end{aligned}
\]

\[
\renewcommand{\arraystretch}{1.6}
\begin{array}{
c@{\qquad}
c@{\qquad}
c@{\qquad}
c
}
\toprule
&
\substack{
\mathcal O_{\bbR,1}\\
=[3^+\,2^{2p-2}\,1^-\,1^{+,2}]
}
&
\substack{
\mathcal O_{\bbR,2}\\
=[3^-\,2^{2p-2}\,1^{+,3}]_{\mathrm{I}}
}
&
\substack{
\mathcal O_{\bbR,3}\\
=[3^-\,2^{2p-2}\,1^{+,3}]_{\mathrm{II}}
}
\\
\midrule
\chi_2
&
\Omega_2
&
\Phi_2
&
\Xi_2
\\
\chi_4
&
\Omega_4
&
\Phi_4
&
\Xi_4
\\
\bottomrule
\end{array}
\]

\subsubsection{The case \(\tu{\Spin}(2p+3,2p-1)\)}

The only real form of
\[
\mathcal O_{\bbC}=[3\,2^{2p-2}\,1^3]
\]
whose closure occurs as the real associated variety of a small representation is
\[
\mathcal O_{\bbR}
=
[3^+\,2^{2p-2}\,1^{+,3}].
\]

\[
\renewcommand{\arraystretch}{1.6}
\begin{array}{c@{\qquad}c}
\toprule
&
\substack{
\mathcal O_{\bbR}\\
=[3^+\,2^{2p-2}\,1^{+,3}]
}
\\
\midrule
\chi_2
&
\Phi_2
\\
\bottomrule
\end{array}
\]

\begin{remark}
For \(\tu{\Spin}(2p+2,2p)\), the map
\[
\pi\longmapsto
\bigl(\mathcal O_{\bbR}(\pi),\chi_\pi\bigr)
\]
gives a bijection between the six small representations and the six pairs
\[
\{\mathcal O_{\bbR,1},\mathcal O_{\bbR,2},\mathcal O_{\bbR,3}\}
\times
\{\chi_2,\chi_4\}.
\]
For $\tu G=\tu{\Spin}(2p+1,2p-1)$, both small representations
have the same real associated variety, but they have different genuine
central characters.  Hence the map above gives a bijection
\[
\{\Phi_1,\Omega_2\}
\longrightarrow
\{\mathcal O_{\bbR}\}\times\{\chi_1,\chi_2\}.
\]
 For
\(\tu{\Spin}(2p+3,2p-1)\), its image consists of the single pair displayed
in the table.

\end{remark}
\subsection{Unitarity}

We conclude this section by recording the unitarity of the small
representations considered above.

\begin{prop}\label{p:unitarity-small}
For each of the groups
\[
\tu{\Spin}(2p+1,2p-1),\qquad
\tu{\Spin}(2p+2,2p),\qquad
\tu{\Spin}(2p+3,2p-1),
\]
all representations in \(\prod_{\rho/2}^s(\tu G)\) are unitary.
\end{prop}

\begin{proof}
In \cite[Section~4]{BTs18}, the representations under consideration are
obtained as irreducible summands of the restrictions of the unitary small
representations of nonlinear covers of odd orthogonal groups constructed in
\cite{LS08}. The restriction of a unitary representation to a closed subgroup
remains unitary, and each irreducible direct summand inherits a positive
definite invariant Hermitian form. It follows that all the small
representations considered here are unitary.
\end{proof}

\end{document}